\documentclass[11pt]{article}

\usepackage{amsmath, amsfonts,amssymb, amsthm}
\usepackage{graphicx}
\usepackage{caption}
\usepackage{subcaption}
\usepackage{color}
\usepackage{epsf}
\usepackage{tikz}
\usepackage[left=1.1in, right=1.1in, top=1in, bottom=1.2in]{geometry}
\usepackage{sidecap}
\usepackage[colorlinks=true, linkcolor=blue, filecolor=magenta, urlcolor=cyan, citecolor=gray, unicode]{hyperref}
\usepackage[capitalize]{cleveref}
\usepackage{accents}
\usepackage{mathrsfs}
\usepackage[shortlabels]{enumitem}
\setlist[enumerate,1]{label=(\roman*),font=\itshape}
\setlist{itemsep=0pt, topsep=1pt}

\usepackage{algorithm}
\usepackage[noend]{algpseudocode}

\newcommand{\rioflat}{{(12+\sqrt{8})}/{17}}

\newcommand{\set}[2]{\left\{#1\mathrel{}\middle|\mathrel{}#2\right\}}

\DeclareMathOperator{\Rd}{\overline{Rd}}
\DeclareMathOperator{\ud}{\overline{d}}
\DeclareMathOperator{\ld}{\underline{d}}
\DeclareMathOperator{\ran}{ran}
\DeclareMathOperator{\dom}{dom}

\DeclareMathOperator{\fin}{fin}

\DeclareMathOperator{\nwd}{nwd}

\DeclareMathOperator{\rul}{rul}
\DeclareMathOperator{\degen}{degen}

\newcommand{\mathscripty}{\mathscr}

\newcommand{\sm}{\setminus}

\renewcommand{\subset}{\subseteq}

\newcommand{\NN}{\mathbb{N}}
\newcommand{\N}{\mathbb{N}}

\newcommand{\QQ}{\mathbb{Q}}

\newcommand{\SU}{\mathscripty{U}}

\newcommand{\CF}{\mathcal{F}}
\newcommand{\CT}{\mathcal{T}}
\newcommand{\CU}{\mathcal{U}}
\newcommand{\CR}{\mathcal{R}}
\newcommand{\CI}{\mathcal{I}}

\theoremstyle{plain}
\newtheorem{theorem}{Theorem}[section]
\newtheorem{corollary}[theorem]{Corollary}
\newtheorem{lemma}[theorem]{Lemma}

\newtheorem{proposition}[theorem]{Proposition}

\newtheorem{fact}[theorem]{Fact}
\newtheorem{example}[theorem]{Example}
\newtheorem{problem}[theorem]{Problem}

\theoremstyle{definition}
\newtheorem{definition}[theorem]{Definition}

\newtheorem{conjecture}[theorem]{Conjecture}
\newtheorem{remark}[theorem]{Remark}

\newcommand{\ep}{\epsilon}

\newcommand{\ceiling}[1]{\lceil#1\rceil}

\newcommand{\tbf}[1]{\textbf{#1}}

\newcommand{\Ncap}{N^{\cap}}
\newcommand{\lrang}[1]{\langle #1 \rangle}

\title{Density of monochromatic infinite subgraphs II}
\author{
  Jan Corsten\thanks{Department of Mathematics, London School of Economics and Political Science, London, UK \texttt{jan.corsten92@gmail.com}} \and
  Louis DeBiasio\thanks{Department of Mathematics, Miami University, Oxford, OH, United States \texttt{debiasld@miamioh.edu}. Research supported in part by Simons Foundation Collaboration Grant \# 283194 and NSF grant DMS-1954170.} \and
  Paul McKenney\thanks{Department of Mathematics, Miami University, Oxford, OH, United States \texttt{pmckenney@gmail.com}}
}

\date{\today}

\begin{document}
\maketitle

\begin{abstract}
In 1967, Gerencs\'er and Gy\'arf\'as \cite{GG} proved a result which is considered the starting point of graph-Ramsey theory: In every 2-coloring of $K_n$ there is a monochromatic path on $\lceil(2n+1)/3\rceil$ vertices, and this is best possible. 
There have since been hundreds of papers on graph-Ramsey theory with some of the most important results being motivated by a series of conjectures of Burr and Erd\H os \cite{BE1, BE2} regarding the Ramsey numbers of trees (settled in \cite{Z}), graphs with bounded maximum degree (settled in \cite{CRST}), and graphs with bounded degeneracy (settled in \cite{L}).

In 1993, Erd\H os and Galvin \cite{EG} began the investigation of a countably infinite analogue of the Gerencs\'er and Gy\'arf\'as result: What is the largest $d$ such that in every $2$-coloring of $K_\mathbb{N}$ there is a monochromatic infinite path with upper density at least $d$. Erd\H os and Galvin showed that $2/3\leq d\leq 8/9$, and after a series of recent improvements, this problem was finally solved in \cite{CDLL} where it was shown that $d={(12+\sqrt{8})}/{17}$.  

This paper begins a systematic study of quantitative countably infinite graph-Ramsey theory, focusing on infinite analogues of the Burr-Erd\H{o}s conjectures.  We obtain some results which are analogous to what is known in finite case, and other (unexpected) results which have no analogue in the finite case.  
\end{abstract}

\section{Introduction}\label{sec:intro}

It was proven by Ramsey \cite{Ramsey1929} that for every graph $G$ and every positive integer $r$, there exists a positive integer $N$ such that in every $r$-coloring\footnote{throughout the paper, an $r$-coloring of a graph $K$ will always mean an $r$-coloring of the edges of $K$} of $K_N$ there is a monochromatic copy of $G$. The smallest possible choice for $N$ is called the \emph{$r$-color Ramsey number} and is denoted by $R_r(G)$. Determining Ramsey numbers of different (families of) graphs is one of the central topics in combinatorics. In this paper, we are interested in similar problems for countably infinite graphs. (We will not consider uncountably infinite graphs and thus always mean ``countably infinite'' when we write ``infinite'' from now on.)

Let $K_{\NN}$ be the graph on vertex set $\N$ with edge set edge set $\binom{\N}{2}$ (typically $\N$ denotes the set of positive integers and we typically begin counting at 1; however, there are certain situations where it is convenient to let $\N$ denote the non-negative integers or to start counting at 0, but this distinction will never have an impact on the results). Ramsey \cite{Ramsey1929} also proved that in every $r$-coloring of $K_\N$, there is a monochromatic copy of $K_\N$.  Thus in order to make the problem quantitative, we will thus consider the density of the monochromatic graphs we are looking for.
The \emph{upper density} of a graph $G$ with $V(G) \subset \N$ is defined as
\[\ud(G) = \limsup_{t\rightarrow\infty} \frac{|V(G) \cap \{0,1,2,\ldots,t\}|}{t}.\]
The \emph{lower density}, denoted $\ld(G)$, is defined analogously in terms of the $\liminf$ and we speak of the \emph{density}, whenever lower and upper densities coincide.

Erd\H{o}s and Galvin~\cite{EG} described a $2$-coloring of $K_{\N}$ in which every graph having finitely many isolated vertices and bounded maximum degree has lower density $0$, thus we typically restrict our attention to upper densities.  However, this does raise the question of whether there is any graph $G$ (with finitely many isolated vertices) having the property that in every 2-coloring of $K_{\NN}$ there is a monochromatic copy of $G$ with positive lower density.  We will return to this question later and prove that, surprisingly, such graphs exist in a strong sense.

Given a graph $G$ and an $r$-coloring of $\phi$ of $K_{\NN}$, the \emph{Ramsey upper density of $G$ with respect to $\varphi$}, denoted $\Rd_{\varphi}(G)$, is the supremum of $\ud(G)$ over all monochromatic copies of $G$ in the coloring $\varphi$ of $K_{\NN}$.  \emph{The $r$-color Ramsey upper density of $G$}, denoted $\Rd_r(G)$, is the infimum of $\Rd_{\varphi}(G)$ over all $r$-colorings $\varphi$ of $K_{\NN}$.  If $r=2$, we drop the subscript.

Possibly the first such (implicitly) quantitative result is due to Rado~\cite{Rado1978} who proved that every $r$-coloring of $K_{\N}$ contains $r$ vertex-disjoint monochromatic infinite paths which together cover all of $\N$.  In particular, one of the paths must have upper density at least $1/r$ and hence $\Rd_r(P_\infty) \geq 1/r$, where $P_\infty$ is the (one-way) infinite path.
For two colors, this was improved by Erd\H{o}s and Galvin~\cite{EG} who proved that $2/3 \leq \Rd(P_\infty) \leq 8/9$.
More recently, DeBiasio and McKenney~\cite{DM} improved the lower bound to $3/4$ and conjectured the correct value to be $8/9$.
Progress towards this conjecture was made by Lo, Sanhueza-Matamala and Wang~\cite{LSW}, who raised the lower bound to $(9+\sqrt{17})/16\approx 0.82019$.  Corsten, DeBiasio, Lamaison, and Lang~\cite{CDLL} finally proved that $\Rd(P_\infty) = \rioflat\approx 0.87226$, thereby settling the problem for two colors.
In this paper, we initiate a systematic study of Ramsey densities for other infinite graphs.  An independent systematic study was undertaken by Lamaison \cite{Lamaison2020}, who fortunately focused on a different aspect of the general problem (locally-finite graphs) and thus the two papers have very little overlap.

\subsection{Graphs with positive Ramsey upper density}
The problem of estimating the Ramsey numbers of sparse finite graphs has received a lot of attention. The problem was motivated by a series of conjectures proposed by Burr and Erd\H os \cite{BE1, BE2}, starting with graphs of bounded maximum degree.

\begin{conjecture}[Burr--Erd\H os \cite{BE1}]\label{conj:linRam1}
	For all $\Delta\in \N$, there exists some $c = c(\Delta)>0$ such that every $2$-colored $K_n$ contains a monochromatic copy of every graph $G$ with at most $cn$ vertices and $\Delta(G)\leq \Delta$.
\end{conjecture}

\cref{conj:linRam1} was solved by Chvat\'al, R\"odl, Szemer\'edi, Trotter \cite{CRST} in an early application of the regularity lemma. Since then, there has been many improvements to the constant $c(\Delta)$ (see \cite{CFS} for a more detailed history). Allen, Brightwell and Skokan~\cite{ABS} proved that this constant can be significantly improved to $ c = 1/(2\chi(G)+4)\geq 1/(2\Delta+6)$ for graphs of small bandwith (see \cite{ABS} for the precise statement of their result), where $\chi(G)$ denotes the chromatic number of $G$.

Our first theorem proves an analogue of this for infinite graphs. It turns out that much weaker conditions on the degrees suffice.
Given $k\geq 2$, we say that a graph $G$ is one-way $k$-locally finite if there exists a partition of $V(G)$ into $k$ independent sets $V_1, \dots, V_k$ with $|V_1|\geq \ldots \geq |V_k|$ such that for all $1\leq i<j\leq k$ and all $v\in V_j$, $d(v, V_i)<\infty$.
Note that every vertex in $V_k$ has finite degree, but it is possible for any vertex in $V_1\cup \dots \cup V_{k-1}$ to have infinite degree.  A good example of a one-way 2-locally finite graph exhibiting this property is the \emph{infinite bipartite half graph}, which is the graph on $\N = A \cup B$, where $A$ is the set of positive odd integers and $B$ is the set of positive even integers and $uv$ is an edge if and only if $u<v$ and $u$ is odd and $v$ is even. Further note that one-way $k$-locally finite graphs have chromatic number at most $k$ and, if $G$ is locally finite (that is every vertex has finite degree) with $\chi(G)<\infty$, then $G$ is one-way $\chi(G)$-locally finite.

\begin{theorem}\label{thm:locally-finite}
Let $k,r\in \N$ and let $G$ be an infinite, one-way $k$-locally finite graph.
\begin{enumerate}
\item If $k=2$, then $\Rd_r(G)\geq 1/r$.
\item If $k\geq 2$, then $\Rd(G)\geq \frac{1}{2(k-1)}$.
\item If $r, k \geq 3$, then
$\Rd_r(G) \geq \frac{1}{\sum_{i=0}^{(k-2)r+1} (r-1)^i} \geq  \frac{1}{r^{(k-2)r + 1}}.$
\end{enumerate}
\end{theorem}

Since graphs with $\Delta(G)=\Delta<\infty$ have $\chi(G)\leq \Delta+1$, we get that $\Rd(G) \geq \frac{1}{2\Delta}$ and $\Rd_r(G) \geq 1/r^{\Delta r}$ for every $r \geq 3$ (which answers a question from \cite{DM}).  However, we are able to prove a slightly stronger result for 2 colors.

\begin{corollary}
If $G$ is an infinite graph with $\Delta(G)=\Delta<\infty$, then $\Rd(G) \geq \frac{1}{2(\Delta-1)}$.
\end{corollary}

A graph $G$ is \emph{$d$-degenerate} if there is an ordering of the vertices $v_1, v_2, \dots, v_n$ such that for all $i\geq 1$, $|N(v_i)\cap \{v_1,\dots, v_{i-1}\}|\leq d$.
The \emph{degeneracy} of $G$, denoted $\degen(G)$, is the smallest non-negative integer $d$ such that $G$ is $d$-degenerate; if no such integer exists, say $\degen(G)=\infty$.
Note that if $G$ is $d$-degenerate, then $\chi(G)\leq d+1\leq \Delta(G)+1$.  Also note that a graph can have finite degeneracy, but infinite maximum degree.

\begin{conjecture}[Burr--Erd\H{o}s \cite{BE1}]\label{conj:linRam2}
  For all $d\in \N$, there exists some $c = c(d) >0$ such that every $2$-colored $K_n$ contains a copy of every $d$-degenerate graph on at most $cn$ vertices.
\end{conjecture}

\cref{conj:linRam2} was recently confirmed by Lee \cite{L}.
It would be very interesting to prove an analogue of this for infinite graphs.

\begin{problem}\label{qu:degenerate}
For all $d\in \N$, does there exist some $c=c(d)>0$ such that $\Rd(G) \geq c$ for every infinite graph $G$ with degeneracy at most $d$?  A weaker version of this question is for all infinite graphs $G$ with finite degeneracy, does there exist some $c=c(G)>0$ such that $\Rd(G) \geq c$?
\end{problem}

As we will discuss in the next section, we obtain a positive answer to a weaker version of this question.

\subsection{Ramsey-dense graphs}
We say that an infinite graph $G$ is \emph{$r$-Ramsey-dense} if in every $r$-coloring of $K_\NN$ there is a monochromatic copy of $G$ with positive upper density.  If $r=2$, we drop the prefix and just say $G$ is \emph{Ramsey-dense}. Note that if $G$ is Ramsey-dense, this does not necessarily imply that $\Rd(G)>0$ as there are infinitely many colorings, so the infimum of the upper densities over all colorings can be $0$.
Indeed, we shall see below that the so called Rado graph $\CR$ is an example of an infinite graph which is Ramsey-dense yet $\Rd(\CR) = 0$.
On the other hand, every infinite graph $G$ with $\Rd(G) > 0$ is Ramsey-dense.

Ramsey-dense graphs are another natural analogue of graphs with linear Ramsey number.
We will describe a simple property guaranteeing that a graph is Ramsey-dense and then show that every Ramsey-dense graph is not far from having this property.

A set $X \subset V(G)$ is called \emph{dominating} if every vertex $v \in V(G) \setminus X$ has a neighbor in $X$. We call a set $X \subset V(G)$ \emph{ruling} if $X$ is finite and all but finitely many vertices $v \in V(G) \setminus X$ have a neighbor in $X$.
We say that an infinite graph $G$ is $t$-ruled if there are at most $t$ disjoint minimal ruling sets. The \emph{ruling number} of a graph $G$, denoted by \emph{$\rul (G)$}, is the smallest $t \in \N$ such that $G$ is $t$-ruled; if no such $t$ exists, we say $G$ is infinitely ruled, or $\rul(G)=\infty$.  Equivalently, $\rul(G)$ is the matching number of the hypergraph whose edges are all minimal ruling sets. Note that a graph $G$ is $0$-ruled if and only if there is no finite dominating set and finitely-ruled (i.e.\ $t$-ruled for some $t \in \N$) if and only if there is a finite set $S \subset V(G)$ such that $G[S^c]$ has no finite dominating sets.

\begin{theorem}\label{thm:ruling}
If $G$ is an infinite graph with $\rul(G)<\infty$, then $G$ is $r$-Ramsey-dense for all $r\in \N$.
\end{theorem}

This has a few interesting corollaries.  Since locally finite graphs have ruling number 0, we immediately get the following.

\begin{corollary}
If $G$ is a locally finite, infinite graph, then $G$ is $r$-Ramsey-dense for all $r\in \N$.
\end{corollary}

The \emph{Rado graph} is the graph $\CR$ with vertex-set $\NN$ defined by placing an edge between $m < n$ if and only if the $m$th digit in the binary expansion of $n$ is $1$. The Rado graph has many interesting properties, for example it is isomorphic to the infinite random graph (that is the graph on $\NN$ in which every edge is present independently with probability $1/2$) with probability $1$. It is easy to verify that the Rado graph does not have any finite dominating sets and hence $\rul (\CR) =0$.

\begin{corollary}
The Rado graph $\CR$ is $r$-Ramsey-dense for all $r\in \N$.
\end{corollary}

On the other hand, we will show that $\Rd(\CR) = 0$ (see \cref{cor:Rado-ud-0}).
Another corollary asserts that graphs with bounded degeneracy are Ramsey-dense.

\begin{corollary}
If $G$ is an infinite graph with bounded degeneracy, then $G$ is $r$-Ramsey-dense for all $r\in \N$.
\end{corollary}

By \cref{thm:ruling}, it suffices to show that every $d$-degenerate infinite graph $G$ is $d$-ruled.

\begin{fact}\label{degen_rul}
Let $d\in \mathbb{N}$.  If $G$ is $d$-degenerate, then $\rul(G)\leq d$.
\end{fact}

\begin{proof}
Suppose for contradiction, there is a $d$-degenerate infinite graph $G$ with $\rul (G) > d$ for some $d \in \N$. Let $S_1, \dots, S_{d+1}$ be disjoint minimal ruling sets and let $S_0 \subset V(G) \setminus (S_1 \cup \ldots \cup S_{d+1})$ be the set of vertices which do not have a neighbor in some $S_i$.
Note that $S := S_0 \cup S_1 \cup \ldots \cup S_{d+1}$ is finite. Therefore, there is a vertex $u \in \N \setminus S $ which comes after all vertices in $S$ in a $d$-degenerate ordering of $V(G)$ and hence $\deg(u,S) \leq d$. However, by construction, $u$ has a neighbor in each of $S_1, \ldots, S_{d+1}$, a contradiction.
\end{proof}

\begin{problem}\label{prob:infruleddense}
Is there a Ramsey-dense graph $G$ with $\rul(G)=\infty$?
\end{problem}

If the answer is no, then together with Theorem \ref{thm:ruling}, this would give a complete characterization of Ramsey-dense graphs.  We will give a partial answer to the question by showing that if $\rul(G)=\infty$ and additionally the sizes of the minimal ruling sets do not grow too fast, then $G$ is not Ramsey-dense (see Theorem \ref{rulinglog}).

\subsection{Trees}
Another famous conjecture of Burr and Erd\H os \cite{BE2} concerns the Ramsey number of trees.  A graph is \emph{acyclic} if it contains no finite cycles, a \emph{forest} is an acyclic graph, and a \emph{tree} is a connected acyclic graph. 

\begin{conjecture}[Burr--Erd\H{o}s \cite{BE2}]\label{conj:BE-trees}
Let $n\in \N$ and let $T$ be a tree on at most $\frac{n}{2}+1$ vertices.  Every $2$-colored $K_n$ contains a monochromatic copy of $T$.
\end{conjecture}

\cref{conj:BE-trees} was solved for large $n$ by Zhao \cite{Z}.
The following result provides an analogue of this in infinite graphs and can be seen to be best possible. Note that \cref{thm:locally-finite} already implies that $\Rd(T) \geq 1/2$ for every infinite locally finite forest $T$.

\begin{theorem}\label{thm:trees}
$\Rd(T) \geq 1/2$ for every infinite forest $T$.
\end{theorem}

We further show that $\Rd(T_\infty)=1/2$ where $T_\infty$ is the infinite tree in which every vertex has infinite degree and there are also infinite locally finite trees $T$ with $\Rd(T) = 1/2$ (see \cref{ex:trees}).  
%Also note that Theorem \ref{thm:trees} clearly applies to all infinite forests as well.

Erd\H{o}s, Faudree, Rousseau, and Schelp \cite{EFRC} showed that if $T$ is a tree on more than $\ceiling{3n/4}$ vertices, then there exists a 2-coloring of $K_n$ which contains no monochromatic copy of $T$.  Furthermore they showed that this bound can be acheived by certain trees such as the tree obtained by joining the center of $K_{1,n/4}$ with a path on $n/2-1$ vertices (see also \cite{YL}).  In other words, $3/4$ is the largest proportion of vertices that a single connected graph can cover in an arbitrary 2-coloring of $K_n$.  We now consider an analogous question for infinite graphs.

Say that a graph $G$ is \emph{Ramsey-cofinite} if in every 2-coloring of $K_{\NN}$ there exists a monochromatic copy of $G$ such that $V(G)$ is cofinite.  It is clear that any graph $G$ with infinitely many isolated vertices is Ramsey-cofinite.  Say that a graph $G$ is \emph{Ramsey-lower-dense} if in every $2$-coloring of $K_\NN$ there is a monochromatic copy of $G$ with positive lower density.  As mentioned earlier, Erd\H{o}s and Galvin proved that for any graph $G$ with finitely many isolated vertices and bounded maximum degree then $G$ is not Ramsey-lower-dense, and thus $G$ is not Ramsey-cofinite.

Surprisingly, we show that there exist connected graphs which are Ramsey-cofinite. In fact, we are able to completely characterize all acyclic graphs which are Ramsey-cofinite.  Say that a graph $G$ is \emph{weakly expanding} if for all $k\in \NN$, there exists $\ell\in \NN$ such that for all independent sets $A$ in $G$ with $|A|\geq \ell$ we have $|N(A)|> k$. Say that a graph $G$ is \emph{strongly contracting} if there exists $k\in \NN$ such that for all $\ell\in \N$ there exists an independent set $A$ in $G$ with $|A|\geq \ell$ such that $|N(A)|\leq k$.  Note that every infinite graph is either strongly contracting or weakly expanding.  Finally, let $\CT^*$ be the family of forests $T$ having one vertex $t$ of infinite degree, every other vertex has degree at most $d$ for some $d\in \NN$, $t$ is adjacent to infinitely many leaves and infinitely many non-leaves, and cofinitely many vertices of $T$ have distance at most $2$ to $t$ (in particular, if $T$ is not connected, then $T$ has one infinite component and finitely many finite components).
%We state our result here in the special case of trees (we will prove a slightly more general result which holds for all forests in \cref{sec:cofingen}).

\begin{theorem}\label{thm:cofinitetree}
Let $T$ be a forest.
\begin{enumerate}
\item If $T$ is strongly contracting, has no finite dominating set, and $T \not \in \CT^{*}$, then $T$ is Ramsey-cofinite.
\item If $T$ is weakly expanding, has a finite dominating set, or $T \in \CT^{*}$, then $T$ is not Ramsey-lower-dense (and thus $T$ is not Ramsey-cofinite).
\end{enumerate}
\end{theorem}

To get a better sense of what Theorem \ref{thm:cofinitetree} says in terms of trees, say that a graph $G$ has  \emph{unbounded leaf degree} if for every $\ell\in \NN$, there exists $v\in V(G)$ such that $v$ is adjacent to at least $\ell$ leaves; otherwise, say that $G$ has \emph{bounded leaf degree}.  A tree is strongly contracting if and only if it has unbounded leaf degree, and a tree is weakly expanding if and only if it has bounded leaf degree.

In light of Theorem~\ref{thm:cofinitetree} it would be natural to ask if there is any connected graph $T$ such that there is a spanning monochromatic copy of $T$ in every 2-coloring of $K_{\NN}$; however, this is not possible.  Clearly if $T$ is an infinite star it does not have this property, so suppose $T$ is not an infinite star and 2-color the edges of $K_{\NN}$ by fixing a vertex $v$, coloring all edges incident with $v$ red, and coloring all other edges blue.  Every monochromatic copy of $T$ must be blue and therefore not be spanning.

Completely characterizing all graphs which are Ramsey-cofinite is still an open question and is discussed in Section \ref{sec:cofingen}.

\subsection{Bipartite Ramsey densities}
Gy\'arf\'as and Lehel \cite{GL} and independently Faudree and Schelp \cite{Faudree1975} proved that every $2$-colored $K_{n,n}$ contains a monochromatic path with at least $2\ceiling{n/2}$ vertices (that is, roughly half the vertices of the graph). They further proved that this is best possible. We will prove an analogue of this for infinite graphs. Here, $K_{\N,\N}$ is the infinite complete bipartite graph with one part being all even positive integers and the other part being all odd positive integers.

\begin{theorem}\label{thm:bip-paths}
Every 2-colored $K_{\NN, \NN}$ contains a monochromatic path of upper density at least $1/2$.
\end{theorem}

Pokrovskiy \cite{P} proved that the vertices of every $2$-colored complete bipartite graph $K_{n,n}$ can be partitioned into three monochromatic paths.
Soukup \cite{Souk} proved an analogue of this for infinite graphs which holds for multiple colors: The vertices of every $r$-colored $K_{\NN,\NN}$ can be partitioned into $2r-1$ monochromatic paths.
He also presents an example where this is best possible. However, in his example all but finitely many vertices can be covered by $r$ monochromatic paths. Our next result shows that this is always possible for two colors.

\begin{theorem}\label{thm:bip-paths-partition}
The vertices of every $2$-colored $K_{\NN, \NN}$ can be partitioned into a finite set and at most two monochromatic paths.
\end{theorem}

\cref{thm:bip-paths} is an immediate consequence of \cref{thm:bip-paths-partition}.
We will provide an example which demonstrates that \cref{thm:bip-paths,thm:bip-paths-partition} are best possible (see \cref{ex:bipartite}).
We believe that a similar statement is true for multiple colors.

\begin{conjecture}\label{conj:bip-paths}
Let $r\in \N$ with $r\geq 2$.  The vertices of every $(r-1)$-colored $K_{\NN, \NN}$ can be partitioned into a finite set and at most $r-1$ monochromatic paths.
\end{conjecture}

\cref{ex:bipartite} also shows that \cref{conj:bip-paths} is best possible, if true.  

The main motivation for the above question had to do with a potential relationship to the problem of determining the value of $\Rd_r(P_\infty)$ for $r\geq 3$.  Very recently, Day and Lo \cite{DL} proved a result which implies that if the above conjecture is true (in fact, if a weaker conjecture is true), then for all $r\geq 3$, $\Rd_r(P_\infty)\geq \frac{1}{r-1}$.  In particular, Theorem \ref{thm:bip-paths} combined with their result implies that $\Rd_3(P_\infty)=\frac{1}{2}$.  They also showed that their weaker conjecture is true for $r=3$ and $r=4$ which additionally gives $\Rd_4(P_\infty)=\frac{1}{3}$.

\subsection{Overview}

We begin by summarizing our main results, and then describe where in the paper these results may be found.  

\begin{enumerate}
\item Let $G$ be a countably infinite, (one-way) locally finite graph with chromatic number $\chi < \infty$ (in particular, the infinite bipartite half-graph has this property).  Every $2$-coloring of $K_\mathbb{N}$ contains a monochromatic copy of $G$ with upper density at least $\frac{1}{2(\chi-1)}$.
\item Let $G$ be a countably infinite graph having the property that there exists a finite set $X\subseteq V(G)$ such that $G-X$ has no finite dominating set (in particular, graphs with bounded degeneracy have this property, as does the infinite random graph). Every finite coloring of $K_\mathbb{N}$ contains a monochromatic copy of $G$ with positive upper density.
\item For every countably infinite tree $T$, every $2$-coloring of $K_\mathbb{N}$ contains a monochromatic copy of $T$ of upper density at least $1/2$, and this is best possible.  This is a perfect analogue of the corresponding result in the finite case which says that every 2-colored $K_n$ contains a monochromatic copy of every tree on at most $\frac{n}{2}+1$ vertices.  
\item  There exists connected graphs $G$ such that every 2-coloring of $K_\mathbb{N}$ contains a monochromatic copy of $G$ which covers all but finitely many vertices of $\mathbb{N}$.  In fact, we classify all trees with this property.  This result is particularly surprising in part because it has no analogue in the finite case (since for every connected graph $G$ on more than $\lceil\frac{3n}{4}\rceil$ vertices, there is a 2-coloring of $K_n$ with no monochromatic copy of $G$).  In the process, we prove two results which may have independent interest:  we give a characterization of graphs which are a spanning subgraph of every infinitely connected graph, and a characterization of graphs which can be cofinitely embedded into every graph with infinitely many vertices of cofinite degree.  
\end{enumerate}

In \cref{sec:example} we collect a variety of examples which are used to for instance obtain upper bounds on the upper Ramsey density of certain graphs.  In \cref{sec:embed} we discuss ultrafilters and a general embedding strategy that we will use to prove our results about one-way locally finite graphs in \cref{sec:chromatic} and graphs of bounded ruling number in \cref{sec:ruling}.  In \cref{sec:degen} we prove some additional results about graphs with bounded degeneracy.  In \cref{sec:bipartite} we prove \cref{thm:bip-paths-partition}.  In \cref{sec:trees} we prove \cref{thm:trees,thm:cofinitetree} together with a variety of supporting results which may be of independent interest.  In \cref{sec:coideals} we discuss a more general extension of the notion of a graph being Ramsey-dense.  Finally we end with some open problems in \cref{sec:end}.

\subsection{Notation}
For a positive integer $n$, we let $[n]=\{1,2,\dots, n\}$.

A subset $X$ of an infinite set $Y$ is called cofinite in $Y$ if $Y \setminus X$ is finite. If $Y$ is clear from context we will call $X$ cofinite and write $X^c = Y \setminus X$.  We write $A \subseteq^* B$ to mean that $A\sm B$ is finite.

Given an edge-colored graph $G$ and a color $c$, we write $G_c$ for the spanning subgraph of $G$ with all edges of color $c$. Given a vertex $v \in V(G)$, we define $N(v)$ to be the set of neighbors of $v$ and, given a color $c$, we define $N_c(v) \subset N(v)$ to be the set of vertices which are adjacent to $v$ via an edge of color $c$.  Given $S\subseteq V(G)$, we write $N(S)=\bigcup_{v\in S}N(v)$ and $\Ncap(S)=\bigcap_{v\in S}N(v)$ and, given a color $c$, we define $N_c(S)=\bigcup_{v\in S}N_c(v)$ and $\Ncap_c(S)=\bigcap_{v\in S}N_c(v)$.

If $f$ is a function, we write $\dom f$ and $\ran f$ for the domain and range of $f$ respectively. (This notation is useful because we are often constructing an embedding of a graph $G$ into a graph $H$ and at each step we have a function from some subset of $V(G)$ to some subset of $V(H)$.)

The following well-known fact follows from the definition of upper- and lower- density.  For disjoint sets $A,B\subseteq\NN$, we have
\begin{equation}\label{fact:density}
\ld(A)+\ld(B) \leq \ld(A\cup B) \leq \ld(A)+\ud(B) \leq \ud(A\cup B) \leq \ud(A)+\ud(B).
\end{equation}

\section{Examples}\label{sec:example}

\subsection{Basics}

First we present some examples to get a better understanding how the different parameters discussed in this paper are related.

The \emph{infinite half graph} is the graph on $\N$ such that $uv$ is an edge if and only if $u<v$ and $v$ is even.  
Given a complete bipartite graph $G$ between two disjoint infinite sets $A$ and $B$, the \emph{half graph coloring} of $G$ is obtained by taking a bijection $f$ from $A$ to the odd integers and a bijection $g$ from $B$ to the even integers and coloring an edge $uv$ with $u\in A$ and $v\in B$ red if $g(u)<f(v)$ and blue otherwise.
Note that in this coloring both the red and the blue graph are isomorphic to the infinite bipartite half graph.

The \emph{bipartite Rado graph} is the graph $\CR_2$ with vertex-set $\NN\setminus \{1\}$ defined by placing an edge between $m < n$ if and only if the $m$th digit in the binary expansion of $n$ is $1$ and $m$ and $n$ differ in the first bit (i.e. $m$ and $n$ have different parity).

\begin{figure}[ht]
\begin{center}
\includegraphics[scale=1]{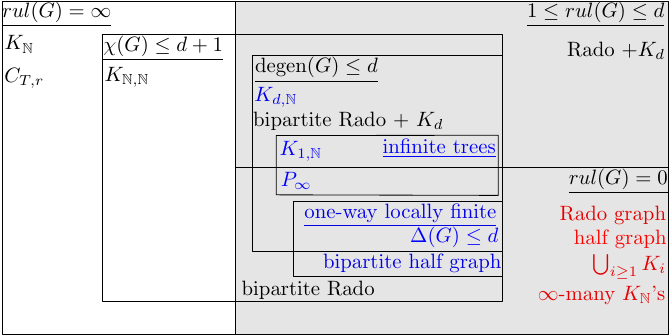}
\end{center}
\caption{The lightly shaded area represents graphs which are Ramsey-dense.  The blue text represents graphs $G$ for which $\Rd(G)>0$.  The red text represents graphs $G$ which are Ramsey-dense, but $\Rd(G)=0$.}
\end{figure}

\begin{example}~
\begin{enumerate}
%\item It is possible to have $rul(G)=0$, but $\chi(G)=\infty$ and $\deg(G)=\infty$ (Rado graph)
\item There is a graph $G$ with $\rul (G) = 0$, but $\chi(G)=\infty$ and thus $\degen(G)=\infty$ (half graph, Rado graph, infinitely many disjoint $K_\N$'s).
%\item It is possible to have $\chi(G)=2$, but $rul(G)=\infty$ and $\deg(G)=\infty$ (complete bipartite graph)
\item There is a graph $G$ with $\chi(G)=2$, but $\rul(G)=\infty$ and thus $\degen(G)=\infty$ ($K_{\N,\N}$).
\item There is a graph $G$ with $\rul(G)=0$ and $\chi (G)=2$, but $\degen(G)=\infty$ (bipartite Rado graph).
\item There is a one-way 2-locally finite graph $G$ (with $\rul(G)=0$ and $\chi (G)=2$), but $\degen(G)=\infty$ (bipartite half graph).
\item There is a locally finite graph $G$ with $\rul (G) =0$ but $\chi(G)=\infty$ and thus $\degen(G)=\infty$ (infinite collection of disjoint finite cliques of increasing size).
\item There is a graph which is $d$-degenerate (and $d$-ruled) but not one-way $k$-locally-finite for any $k$ ($K_{d, \NN}$, $T_\infty$)
\end{enumerate}
\end{example}

\subsection{Upper bounds on upper densities}

\begin{example}\label{ex:trees}
Let $r\in \NN$.
\begin{enumerate}
\item Let $D\geq 2$.  If $T$ is an infinite $D$-ary tree, then $\Rd_r(T)\leq \frac{1}{r}(1+\frac{1}{D})$.

\item There exists a locally finite, infinite tree $T$ such that $\Rd_r(T)\leq 1/r$.
\end{enumerate}
\end{example}

\begin{proof}
Partition $\N$ by residues mod $r$, that is $\N = A_0 \cup \ldots \cup A_{r-1}$ where $A_i=\{n \in \N:  n \equiv i \pmod r\}$. We define an $r$-coloring as follows: if $m \in A_i$ and $n>m$, color the edge $mn$ with color $i$.  Note that if $n\not\equiv i\ (\bmod\ r)$, then $n$ has at most $\ceiling{(n-1)/r}$ neighbors of color $i$.

(i) Let $T$ be an infinite $D$-ary tree and suppose we have a copy of $T$ of color $i$.  For all $n\in \N$, let $V'_n$ be the set of vertices in $V(T)\cap [n]$ which are not congruent to $i \pmod r$.  Since every vertex $m\in V'_n$ can only have successors (of color $i$) in $A_i \cap [n-1]$, we must have $D \cdot |V_n'| \leq \ceiling{(n-1)/r}$. So we have
\[\frac{|V(T)\cap [n]|}{n}\leq \frac{\ceiling{\frac{n}{r}}+|V_n'|}{n}\leq \frac{\ceiling{\frac{n}{r}}+\frac{1}{D}\ceiling{\frac{n-1}{r}}}{n} \xrightarrow{n \to \infty} \frac{1}{r}(1+\frac{1}{D}),\]
and thus $\Rd_r(T)\leq \frac{1}{r}(1+\frac{1}{D})$.

(ii) Let $0<d_1< d_2< \dots$ be an increasing sequence of integers.  Let $T$ be a tree in which every vertex on level $i$ has degree $d_i$.  We can repeat the argument from case (i), except now we have $|V_n'|/n \to 0$ as $n \to \infty$ and thus $\frac{|V(T)\cap [n]|}{n}\leq \frac{\ceiling{\frac{n}{r}}+|V_n'|}{n} \xrightarrow{n \to \infty} \frac{1}{r}$.
\end{proof}

Note that when $r=2$, there are connected graphs $G$ in which every vertex has infinite degree but $\Rd_2(G)=\Rd(G)\geq 1/2$ (\cref{thm:trees} for instance).  The following example shows that there is an unexpected change in behavior as we go from 2 colors to 3 colors.

\begin{example}\label{ex:3ormore}
Let $r\in \N$ with $r\geq 3$.  If $G$ is a graph with finitely many components and finitely many vertices of finite degree, then $\Rd_r(G)=0$.
\end{example}

\begin{proof}
Let $\ep>0$ be given, let $c$ be the number of components of $G$, and let $k$ be an integer with $k>c/\ep$.  Partition $\N$ by residues mod $k$, that is $\N = A_0 \cup \ldots \cup A_{k-1}$ where $A_i=\{n \in \N:  n \equiv i\ (\bmod\ k)\}$.  For all $0\leq i\leq k-1$, color all edges inside $A_i$ with green and for all $0\leq i<j\leq k-1$, color the edges between $A_i$ and $A_j$ with the half graph coloring where the vertices in $A_i$ have cofinite red degree to $A_j$ and the vertices in $A_j$ have cofinite blue degree to $A_i$.  Note that we have only used three colors, but this can be considered as an $r$-coloring for all $r\geq 3$. 

Note that for all $0\leq i\leq k-1$, every vertex in $A_i$ has finite red degree to $A_0\cup \dots \cup A_i$.  If there is a red copy of $G$, then let $0\leq i\leq k-1$ be maximum such that $V(G)\cap A_i$ is infinite.  But this is a contradiction because every vertex in $V(G)\cap A_i$ has finite red degree.  Similarly, for all $0\leq i\leq k-1$, every vertex in $A_i$ has finite blue degree to $A_i\cup \dots \cup A_{k-1}$.   If there is a blue copy of $G$, then let $0\leq i\leq k-1$ be minimum such that $V(G)\cap A_i$ is infinite.  But this is a contradiction because every vertex in $V(G)\cap A_i$ has finite blue degree. Therefore every monochromatic copy of $G$ is green and thus has upper density at most $c/k<\ep$.
\end{proof}

\begin{example}\label{ex:rd-ub-chi}
For every non-trivial connected graph $G$, $\Rd(G)\leq \frac{1}{\chi(G)-1}$.  In particular, if $\chi(G)=\infty$, then $\Rd(G)=0$.
\end{example}

\begin{proof}
Assume first that $\chi(G) < \infty$ and partition $\NN$ by residues mod $\chi(G)-1$. Color all edges inside the sets red and all edges between the sets blue. There is no blue copy of $G$, so every copy of $G$ lies entirely inside one of the sets, all of which have density $\frac{1}{\chi(G)-1}$.

If $\chi(G) = \infty$, this construction shows that $\Rd(G) \leq 1/(k-1)$ for every $k \geq 2$ and therefore $\Rd(G) = 0$.
\end{proof}

\begin{corollary}\label{cor:Rado-ud-0}~
\begin{enumerate}
\item $\Rd(\CR)=0$ (where $\CR$ is the Rado graph).
\item There exists a locally finite graph $G$ such that $\Rd(G)=0$.
\end{enumerate}
\end{corollary}

\begin{proof}
 (i) Since $\CR$ contains an infinite clique, we have $\chi(\CR) = \infty$ and thus the result follows from \cref{ex:rd-ub-chi}.

  (ii) Let $G$ be a graph on vertex set $\NN$ where $[\frac{n(n+1)}{2}, \frac{(n+1)(n+2)}{2}]$ induces a clique for all $n\in \N$.  $G$ is locally finite, connected, and contains a clique of order $n$ for all $n\in \N$.  So $\chi(G)=\infty$ and thus the result follows from \cref{ex:rd-ub-chi}.
\end{proof}

\begin{example}\label{ex:bipartite}
For all $r\in \NN$, there is an $r$-coloring of $K_{\N,\N}$ in which every monochromatic path has upper density at most $1/r$. In particular, it is not possible to cover all but finitely many vertices with less than $r$ monochromatic paths.
\end{example}

\begin{proof}
  Let $A$ and $B$ be the parts of $K_{\N,\N}$ and partition both of them into $r$ parts $A_1, \ldots, A_r$ and $B_1, \ldots, B_r$, each of density $1/(2r)$. For all $i,j \in [r]$, color every edge between $A_i$ and $B_j$ by $(i-j) \mod r$. It is easy to see that every part is incident to exactly one other part of each color and therefore, every monochromatic path can cover at most two parts, finishing the proof.
\end{proof}

\subsection{Lower density}
As mentioned in the introduction, Erd\H{o}s and Galvin proved that for all positive integers $\Delta$, there exists a 2-coloring of
  $K_\NN$ such that if $G$ is a graph with maximum degree at most $\Delta$ and finitely many isolated vertices, then every monochromatic copy of $G$ has lower density 0.  We now show that a broader class of graphs has this property.

Recall that a graph $G$ is weakly expanding if for all $k\in \NN$, there exists $\ell\in \NN$ such that for all independent sets $A$ in $G$ with $|A|\geq \ell$ we have $|N(A)|> k$.  Note that if $G$ is weakly expanding, then there is an increasing function $f:\NN\to \NN$ such that for all $k\in \NN$, if $A$ is an independent set in $G$ with $|A|\geq f(k)$, then $|N(A)|> k$.  Also note that if $G$ is weakly expanding, then $G$ has finitely many isolated vertices.  To better understand this definition, we collect some useful properties which imply that that a graph is weakly expanding. 

\begin{fact}
$G$ is weakly expanding if
\begin{enumerate}
\item $G$ has finite independence number, or
\item $G$ has bounded maximum degree and finitely many isolated vertices, or
\item $G$ is a tree with bounded leaf degree, or
\item for all $n\in \NN$, $G$ has finitely many vertices of degree $n$.
\end{enumerate}
\end{fact}

The following is a modification of the example used by Erd\H{o}s and Galvin \cite{EG} to prove the result mentioned about about graphs with bounded maximum degree and finitely many isolated vertices.

\begin{example}[Forward interval coloring]\label{ex:forward-half-graph}
If $G$ is a graph which is weakly expanding, then $G$ is not Ramsey-lower-dense.
\end{example}

We note that the forthcoming \cref{lem:embed-leafy-tree} shows that if $G$ is strongly contracting, then there is a confinite monochromatic copy of $G$ in every forward interval coloring.

\begin{proof}
Suppose $G$ is weakly expanding and let $f$ be the function guaranteed by the definition.

Let $a_n$ be an increasing sequence of natural numbers such that $a_0=1$ and for all $k\geq 1$,
\begin{equation}\label{ak}
a_k> k( a_{k-1}+f(a_{k-1})).
\end{equation}
For all $u,v\in \NN$ with $u<v$, color the edge $uv$ red if $u\in [a_{2n-1},a_{2n})$ and blue if $u\in [a_{2n},a_{2n+1})$ for some $n\in \NN$.

Suppose there is a, say, blue copy of $G$ in this 2-coloring with vertex set $U$.  We must have that $A_n:=U\cap [a_{2n-1},a_{2n})$ induces an independent set for all $n\in \mathbb{N}$, and because of the coloring we have $N_B(A_n)\subseteq [0, a_{2n-1})$ and thus $|N_B(A_n)|\leq a_{2n-1}$.  Thus by the definition of weakly expanding we have $|A_n|<f(a_{2n-1})$.
We conclude that for all $n\in \NN$, $$|U\cap [0, a_{2n})|=|U\cap [0,a_{2n-1})|+|A_n|<a_{2n-1}+f(a_{2n-1})\stackrel{\eqref{ak}}{<}\frac{1}{2n}a_{2n}$$ and thus $\ld(U)=0$.
\end{proof}

We conclude with two more examples.

\begin{example}[Backward interval coloring]\label{ex:backward-half-graph}
If $G$ is a graph with a finite dominating set (i.e.\ $\rul(G)>0$), then $G$ is not Ramsey-lower-dense.
\end{example}

We note that the forthcoming \cref{prop:embed0rul} shows that if $G$ has no finite dominating set, then there is a spanning monochromatic copy of $G$ in every backwards interval coloring.

\begin{proof}
Let $a_n$ be an increasing sequence of natural numbers and let $A_i=[a_i, a_{i+1})$ for all $i\in \NN$.  For all $u\in A_i$ and $v\in A_j$ with $u<v$, color the edge $uv$ red if $j$ is odd and blue if $j$ is even.  Let $A^0$ be the union of all even indexed intervals and let $A^1$ be the union of all odd indexed intervals.  We note that every vertex in $A^0$ has finite blue degree to $A^1$ and every vertex in $A^1$ has finite red degree to $A^0$.

Let $D$ be a finite dominating set in $G$ and suppose there is a monochromatic, say, blue copy of $G$ with vertex set $V$.  Since $D$ is finite, there exists an index $t$ such that $D\subseteq A_1\cup A_2\cup \dots \cup A_t$.  Now for all $i$ such that $2i+1>t$, there are no blue edges from $A_{2i+1}$ to $D$ contradicting the fact that $D$ is a dominating set.  So $G$ has finite intersection with say $A^1$ and thus if $a_n$ is increasing fast enough, $G$ has lower density 0.
\end{proof}

\begin{example}\label{ex:KNN}
Let $G$ be a connected graph.  If $\chi(G)\geq 3$, or $G$ is bipartite with one part finite, then $G$ is not Ramsey-lower-dense.
\end{example}

\begin{proof}
Let $\{X, Y\}$ be a partition of $\NN$ into two sets of lower density 0 (for instance, as we did in Example \ref{ex:forward-half-graph} and Example \ref{ex:backward-half-graph}).  Color all edges inside $X$ or inside $Y$ with blue and color all edges between $X$ and $Y$ red.  Note that since $G$ is connected, any blue copy of $G$ is completely contained in $X$ or $Y$ and thus has lower density 0.  

If $\chi(G)\geq 3$, then there is no red copy of $G$ and we are done.  If $G$ is bipartite and one of the parts is finite, then $G$ intersects either $X$ or $Y$ in only finitely many vertices and thus any red copy of $G$ will have lower density 0.  
\end{proof}

\subsection{The Rado graph, \texorpdfstring{$0$}{0}-ruled and \texorpdfstring{$0$}{0}-coruled graphs}\label{sec:rado}

If $G$ and $H$ are two graphs, then we write $G \preceq H$ if $G$ is isomorphic to a spanning subgraph of $H$.  Clearly, $\preceq$ is reflexive and transitive.

  We say that an infinite graph $G$ has the \emph{extension property} if for every pair of disjoint finite sets $F,F'\subseteq V(G)$, there is a vertex $v\in V(G)\sm (F\cup F')$ such that $v$ is adjacent to every $w\in F$ and not adjacent to any $w'\in F'$. The following well-known theorem (see \cite{Cam}) shows why this property is useful.

  \begin{theorem}\label{thm:ext-prop}
    Any two infinite graphs satisfying the extension property are isomorphic.
  \end{theorem}

  Furthermore, it is not hard to see that the Rado graph $\CR$ and (with probability $1$) the infinite random graph (every edge is present independently with probability $1/2$) both satisfy the extension property. Hence, with probability $1$, the infinite random graph is isomorphic to the Rado graph.

  Observe that $G$ is $0$-ruled if and only if $G$ satisfies the ``non-adjacency'' half of the extension property above, i.e.\ if for every finite $F'\subseteq V(G)$ there is a vertex $v\in V(G)\sm F'$ such that $v$ is not adjacent to any $w'\in F'$.
  We will call $G$ \emph{$0$-coruled} if $G$ satisfies only the ``adjacency'' half of extension property, i.e.\ for every finite $F\subseteq V(G)$ there is a $v\in V(G)\sm F$ such that $v$ is adjacent to every $w\in F$. Using this, it is easy (and very similar to the proof of \cref{thm:ext-prop}) to prove the following proposition.
  %The proof of the following proposition is essentially the same as the proof of the uniqueness of the random graph.

  \begin{proposition}\label{prop:0-ruled-coruled}
    $G$ is $0$-ruled if and only if $G\preceq \CR$.  On the other hand, $G$ is
    $0$-coruled if and only if $\CR\preceq G$.
  \end{proposition}

Note that for finite graphs $G\preceq H$ and $H\preceq G$ implies $G\cong H$, and thus $\preceq$ is a partial order (on isomorphism classes of graphs), but this is not the case for infinite graphs.  A simple example is letting $G$ be an infinite clique together with infinitely many disjoint copies of some finite graph $F$ and letting $H$ be two disjoint infinite cliques together with infinitely many disjoint copies of some finite graph $F$). Another example comes from the fact that the infinite half graph is both $0$-ruled and $0$-coruled, but is not isomorphic to $\CR$. We ask the following question out of curiosity\footnote{We asked this question on MathOverflow and received evidence \cite{MO1} suggesting that there may not be a simple answer.}.

\begin{problem}
Under what conditions on $G$ and $H$ does $G\preceq H$ and $H\preceq G$ imply that $G\cong H$?
\end{problem}

%The following is just a few more examples of 0-ruled graphs.
%
%\begin{example}~
%\begin{enumerate}
%%\item $\rul(K_{\NN, \NN})=\infty$.
%
%%\item $\rul(K_{t, \NN})=t$.
%
%\item The infinite half graph is $0$-ruled.
%
%\item The union of infinitely many infinite cliques is $0$-ruled.
%\end{enumerate}
%\end{example}

The \emph{Rado coloring} of $E(K_{\N})$ is the $2$-coloring $\rho$ defined by setting $\rho(\{s,t\})$ to be the $s$th bit in the binary expansion of $t$ for all $s,t\in \N$ with $s<t$.
For instance $\rho(\{2,14\})=1$ since the 2nd bit (reading right to left) in the binary expansion of 14 is 1, and $\rho(\{5,14\})=0$ since the 5th bit (reading right to left, and appending extra 0's to the left as necessary) in the binary expansion of 14 is 0.
Also note that the Rado coloring can be described by coloring all of the edges of the Rado graph with color 1 and coloring all of the edges in the complement of the Rado graph with color 0.

The key property of the Rado coloring is that for any
    $F\subseteq\NN$ and $i\in\{0,1\}$, we have
    \begin{equation}\label{eq:rado}
      d\left(\bigcap_{v\in F} N_i(v)\right) = 2^{-|F|}.
    \end{equation}

%
%  \begin{proof}
%    Modulo a finite set, the neighborhood $N$ of $F$ consists of those integers
%    which have binary representation
%    \[
%      \ast\cdots\ast 1 \ast\cdots \ast 1 \cdots 1 \ast \cdots
%    \]
%    where the places of the $1$'s are the members of $F$, and $\ast$ indicates a
%    $0$ or a $1$.  Let $a = \max(F)$, and $k = |F|$; then on each interval of
%    the form $[n, n + 2^a]$, where $n$ is divisible by $2^a$, $N$ consists of
%    the same pattern of $2^{a-k}$ integers.  Hence the asymptotic density of $N$
%    is $2^{a-k} / 2^a = 2^{-k}$.
%  \end{proof}

We first use the Rado coloring to make the following observation about complete multipartite graphs.

\begin{proposition}\label{prop:multi}
Let $K$ be an infinite complete multipartite graph and let $n\in \N$.
\begin{enumerate}
\item If $K$ has at least two infinite parts, or infinitely many vertices in finite parts, then $K$ is not Ramsey-dense.

\item If $K$ has exactly one infinite part and exactly $n$ vertices in finite parts, then $$\frac{1}{2^{2n-1}}\leq \Rd(K)\leq \frac{1}{2^{n}}.$$
\end{enumerate}

\end{proposition}

\begin{proof}
Take the Rado coloring of $K_{\NN}$.

If $K$ has at least two infinite parts, or infinitely many vertices in finite parts, then $K$ contains a spanning copy of $K_{\NN, \NN}$; let $(A,B)$ be such a spanning copy of $K_{\NN, \NN}$.  Let $a_1,a_2,\ldots$ be the elements of $A$.  Then $B$ is contained in the neighborhood of $a_1,\ldots,a_n$, and hence has density $\le 2^{-n}$, for
    each $n$, by \eqref{eq:rado}.  Hence $B$ must have density $0$.  The same goes for $A$.

Now suppose $K$ has exactly one infinite part and exactly $n$ vertices in finite parts. Then by \eqref{eq:rado}, we have $\Rd(K)\leq \frac{1}{2^{n}}$ since the infinite part is the intersection of the neighborhoods of the $n$ vertices in finite parts.

To see $\Rd(K)\geq \frac{1}{2^{2n-1}}$, we are given an arbitrary 2-coloring of $K_{\NN}$ and we choose an arbitrary vertex $v$.  Either $\ud(N_R(v))\geq 1/2$ or $\ud(N_B(v))\geq 1/2$ and we choose the color with largest upper density.  We repeat this process inside the chosen neighborhood for $2n-1$ steps, at which point we have $n$ vertices whose common neighborhood in color, say, red has upper density at least $\frac{1}{2^{2n-1}}$ and the infinite part of $K$ can be embedded in such a way that it spans this set.
\end{proof}

The following result suggests that the question of whether $G$ is Ramsey-dense or not may
depend on the rate of growth of the ruling sets in $G$.

\begin{theorem}\label{rulinglog}
  Let $G$ be an infinite graph.  If $G$ has pairwise-disjoint ruling sets
  $F_n$ ($n\in\NN$) satisfying $|F_n| \le \log_2(n)$ for all sufficiently
  large $n$, then $G$ is not Ramsey-dense.
\end{theorem}

\begin{proof}
Consider the Rado coloring of $K_\NN$.
Suppose now that $V$ is the vertex set of a monochromatic copy of $G$, say
with color $i$.  Then for each $N$, we have
\[
  V \subseteq^* \bigcap_{n=1}^N \bigcup_{v\in F_n} N_i(v).
\]

Note that
\[
  d\left(\bigcap_{n=1}^N \bigcup_{v\in F_n} N_i(v)\right) = \prod_{n=1}^N (1
  - 2^{-|F_n|})
\]
Hence
\[
  \ud(V)\le \prod_{n=1}^\infty (1 - 2^{-|F_n|}).
\]
It is well-known that an infinite product $\prod_{n=1}^\infty \alpha_n$,
with $\alpha_n\in (0,1)$, converges to $0$ if and only if
\[
  \sum_{n=1}^\infty \log(\alpha_n) = -\infty.
\]
In our case we have $|F_n| \le \log_2(n)$ for all sufficiently large $n$, so
\[
  \log(1 - 2^{-|F_n|}) \le \log\left(1 - \frac{1}{n}\right) \leq -1/n.
\]
By the limit comparison test and the divergence of the harmonic series, we have $d(V) = 0$.
\end{proof}
%%%%%%%%%%%%%%%%%%%%%%%%%%%%%%%%%%%%%%%%%%%%%%%%%%%%%%%%%%%%
%%%%%%%%%%%%%%%%%%%%%%%%%%%%%%%%%%%%%%%%%%%%%%%%%%%%%%%%%%%%
\section{Ultrafilters and embedding}\label{sec:embed}
The concept of ultrafilters will play an important role in this paper.

\begin{definition}
  \label{def:ultrafilter}
  Given a set $X$, a set system $\SU \subset 2^X$ is called an \emph{ultrafilter} if
  \begin{enumerate}
    \item $X \in \SU$ and $\emptyset \not \in \SU$,
    \item If $A \in \SU$ and $A \subset B \subset X$, then $B \in \SU$,
    \item If $A,B \in \SU$, then $ A \cap B \in \SU$ and
    \item For all $A \subset X$, either $A \in \SU$ or $X \setminus A \in \SU$, or
    \item[(iv)$'$] $\SU$ is maximal among all families satisfying $(i)$ - $(iii)$.
  \end{enumerate}
\end{definition}

A family satisfying $(i)$-$(iii)$ is called a \emph{filter}. Conditions $(iv)$ and $(iv)'$ are equivalent for filters (see \cite[Chapter 11, Lemma 2.3]{Hrbacek1999}) and we will make use whichever is more convenient for the current application.
Let us list some additional properties of ultrafilters.

\begin{proposition}
  \label{prop:ultrafilter-properties}
 If $\SU$ is an ultrafilter on $X$, we have
  \begin{enumerate}
    \item If $A_1, \ldots, A_n \in \SU$, then $A_1 \cap \ldots \cap A_n \in \SU$.
    \item If $A_1, \dots, A_n$ are pairwise disjoint and $A_1 \cup \ldots \cup A_n \in \SU$, then there is exactly one $i \in [n]$ with $A_i \in \SU$.
  \end{enumerate}
\end{proposition}

Informally, we think of sets $A \in \SU$ as ``large'' sets. A common example of an ultrafilter are the so called trivial ultrafilters $\SU_x := \{A \subset X: x \in A\}$ for $x \in X$. It is not hard to see that an ultrafilter is trivial if and only if it contains a finite set.

We say that an ultrafilter $\SU$ on $\N$ is \emph{positive} if every set $A \in \SU$ has positive upper density in $\NN$. Positive ultrafilters play a crucial role in the proof of Theorem \ref{thm:ruling}.

\begin{proposition}\label{prop:non-trivial-uf}
  If $X \subset \N$ is infinite, then there exists a non-trivial ultrafilter $\SU$ on $X$.
  There exists a positive ultrafilter $\SU$ on $\N$.
  \label{prop:positive_ultrafilter}
\end{proposition}

\begin{proof}
  To prove the first part of the theorem apply Zorn's lemma to
\[\{\CF \subset 2^\N: \CF \text{ contains all cofinite sets and  satisfies (i) - (iii) in \cref{def:ultrafilter}}\}\] to get a maximal such family $\SU$, which must be an ultrafilter. Finally, if $A$ is finite, $\SU$ contains the cofinite set $A^c$ and hence $A \not \in \SU$.

To prove the second part, apply Zorn's lemma to
\[\{\CF \subset 2^\N: \CF \text{ contains all sets of lower density } 1 \text{ and satisfies (i) - (iii) in \cref{def:ultrafilter}}\}\] to get a maximal such family $\SU$, which must be an ultrafilter. Furthermore, if $A \subset \N$ has upper density $0$, then $\N \setminus A$ has lower density $1$ (see \eqref{fact:density}) and consequently $A \not \in \SU$.
\end{proof}

\begin{definition}[Vertex-coloring induced by $\SU$]
  Let $r\geq 2$ be an integer and suppose the edges of an infinite graph $G$ are colored with $r$ colors. Let $\SU$ be a non-trivial ultrafilter on $V(G)$.
  Define a coloring $c_\SU: V(G) \to [r]$ where $c_\SU(v)=i$ if and only if $N_i(v)\in \SU$. Since $V(G) \setminus \{v\} \in \SU$ for all $v \in V(G)$, it follows from \cref{prop:ultrafilter-properties} (ii) that $c_\SU$ is well-defined.
  We call $c_\SU$ the vertex-coloring induced by $\SU$.
\end{definition}

The following two propositions allow us to use ultrafilters to embed the desired subgraphs in the proof of \cref{thm:locally-finite} and \cref{thm:ruling}.

\begin{proposition}\label{prop:embed}
Let $k\geq 2$ be an integer, let $G$ be a one-way $k$-locally finite graph and let $H$ be a graph such that $\{U_1,\dots, U_k\}$ is a partition of $V(H)$ with $|U_1| = \dots = |U_k| = \infty$ and for all $i\in [k]$ and any finite subset $W\subseteq U_1\cup \dots \cup U_{i-1}$, the set of common neighbors of $W$ in $U_i$ is infinite.
Then, there is an embedding $f$ of $G$ into $H$ such that $U_1\subseteq \ran f$.
\end{proposition}

Given a $k$-partite graph $G$ with parts $V_1, \dots, V_k$ and a set $S\subseteq V(G)$, the \emph{left neighborhood cascade} of $S$ is the tuple $(S_1, \ldots, S_k)$, where $S_{k}=S\cap V_k$, and for all $1\leq i\leq k-1$, $S_i= (S\cup \bigcup_{j=i+1}^k N(S_{j})) \cap V_i$.

\begin{proof}
%[Proof of \cref{prop:embed}]
Let $V_1\cup V_2\cup \dots \cup V_k$ be a partition of $V(G)$ into independent sets which witness the fact that $G$ is one-way $k$-locally-finite (in particular $V_1$ is infinite).
We will construct an embedding $f$ iteratively in finite pieces. Initially, $f$ is the empty embedding. Then, for each $n \in \N$, we will proceed as follows: let
\[S_n=\{\min (V_i \setminus \dom f): i \in [k] \text{ with } V_i \setminus \dom f \not= \emptyset\}.\]
That is, $S_n$ contains the smallest not yet embedded vertex of each $V_i$ which is not completely embedded yet.
Let $(T_{1,n}, \ldots, T_{k,n})$ be left neighborhood cascade of $S_n$ in $G$. We will now extend $f$ to cover $\bigcup_{i \in [k]} T_{i,n}$.
Observe that $T_{i,n}$ is disjoint from $\dom f$ for all $i \in [k]$ since we embedded the whole left neighborhood cascade in every previous step.
Since $V_1$ is infinite, $T_{1,n}$ is non-empty. Let $T'_{1,n} \subset U_1 \setminus \ran f$ be the set of $|T_{1,n}|$ smallest vertices in $U_1 \setminus \ran f$ and extend $f$ by embedding $T_{1,n}$ into $T'_{1,n}$ arbitrarily.
By assumption $T'_{1,n}$ has infinitely many common neighbors in $U_2$. Since $\ran f$ is finite, we can select a set $T_{2,n}' \subset (U_2 \cap N^{\cap}(T'_{1,n}) \setminus \ran f$ of size $|T_{2,n}|$. Extend $f$ by embedding $T_{2,n}$ into $T'_{2,n}$ arbitrarily. Similarly, we can extend $f$ by embedding $T_{i,n}$ into appropriate sets $T'_{i,n}$ for all $i=3, \ldots, k$.

Since we maintain a partial embedding of $G$ into $H$ throughout the process and every vertex of $G$ will eventually be embedded (by choice of $S_n$ which contains the smallest not yet embedded vertex of $V(G)$), the resulting function $f$ defines an embedding of $G$ into $H$. Since we cover the smallest not-yet covered vertex of $U_1$ in each step, we further have $U_1 \subset \ran f$.
\end{proof}

\begin{proposition}\label{prop:embed0rul}
Let $H$ be a graph having the property that for every finite set of vertices $W\subseteq V(H)$, the set of common neighbors of $W$ is infinite.  If $G$ is an infinite $0$-ruled graph, then there is a surjective embedding of $G$ into $H$.
\end{proposition}

\begin{proof}
%[Proof of \cref{prop:embed0rul}]
  Let $v_1,v_2, \dots$ be an enumeration of $V(G)$ and let $u_1, u_2,\dots$ be an enumeration of $V(H)$.  Let $f(v_1)=u_1$.  Now suppose $\dom f=\{v_1,\dots, v_{n}\}$ for some $n \in \NN$.  Let $u_{i_n}$ be the vertex of smallest index in $V(H)\setminus \ran f$.
  Since $G$ is $0$-ruled, there exists a vertex $v_p$ with $p> n$ such that $v_p$ has no neighbors in $\{v_1, \ldots, v_{n}\}$. We set $f(v_p) = u_{i_n}$ and if $p>n+1$, we do the following for all $n+1\leq i\leq p-1$:
  since $\{f(v_1), \dots, f(v_{i-1}), f(v_p)\}$ has infinitely many common neighbors, we may choose a vertex $u\in V(H)\setminus \ran f$ which is adjacent to every vertex in $\{f(v_1), \dots, f(v_{i-1}), f(v_p)\}$ and set $f(v_i)=u$.
  Continuing in this way, we obtain an embedding of $G$ into $H$.  Since on each step, the vertex of smallest index in $V(H)\setminus \ran f$ becomes part of the range of $f$, the embedding is surjective.
\end{proof}
%%%%%%%%%%%%%%%%%%%%%%%%%%%%%%%%%%%%%%%%%%%%%%%%%%%%%%%%%%%%%%%%%%
%%%%%%%%%%%%%%%%%%%%%%%%%%%%%%%%%%%%%%%%%%%%%%%%%%%%%%%%%%%%%%%%%%
%%%%%%%%%%%%%%%%%%%%%%%%%%%%%%%%%%%%%%%%%%%%%%%%%%%%%%%%%%%%%%%%%%
\section{Graphs of bounded chromatic number}\label{sec:chromatic}
In this section we will prove \cref{thm:locally-finite}.  First note that if $G$ is one-way $k$-locally-finite, then $G$ is 0-ruled.

\begin{proof}[Proof of \cref{thm:locally-finite}]
(i) We are given an infinite one-way $2$-locally-finite graph $G$ and an $r$-coloring of the edges of $K_\NN$.  Let $\SU$ be a non-trivial ultrafilter on $\NN$.  Let $c_\SU$ be the vertex-coloring induced by $\SU$ and for all $i\in [r]$, let $A_i$ be the set of vertices receiving color $i$.  We may suppose without loss of generality that $\ud(A_1)\geq 1/r$ (see \eqref{fact:density}).
If $A_1\in \SU$, then the set of common neighbors of $S$ in $A_1$ of color $1$ is infinite for every finite set $S \subset A_1$. Thus, we can apply \cref{prop:embed0rul} to embed $G$ in color $1$ in such a way that $A_1$ is covered. If $A_1 \not\in \SU$, then every finite set $S \subset A_1$ has infinitely many common neighbors of color $1$ in $A_1^c$. Hence, by applying Proposition \ref{prop:embed}, we can find a monochromatic copy of $G$ in color $1$ such that $A_1$ is covered. Either way, we have a monochromatic copy of $G$ of upper density at least $\ud(A_1) \geq 1/r$.

\begin{figure}[ht]
  \centering
  \begin{subfigure}[t]{0.5\linewidth}
    \centering\includegraphics[scale=1]{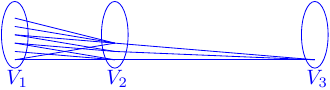}
    \caption{A one-way 3-locally-finite graph $G$\label{fig:fig1}}
  \end{subfigure}%
  \begin{subfigure}[t]{0.5\linewidth}
    \centering\includegraphics[scale=1]{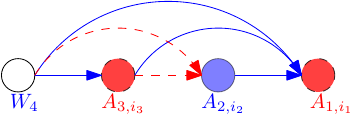}
    \caption{An arrow from $X$ to $Y$ of color $i$ indicates that any finite set of vertices in $X$ has infinitely many common neighbors in $Y$ of color $i$.  If $X$ is filled with color $i$, then any finite set of vertices in $X$ has infinitely many common neighbors in $X$ of color $i$. \label{fig:fig2}}
  \end{subfigure}
  \caption{An example of the proof of \cref{thm:locally-finite}.(ii).  In this example, $G$ will be embedded in blue into $W_4\cup A_{3,i_3}\cup A_{1,i_1}$ such that $W_4\subseteq V(G)$.}
\end{figure}

(ii) We are given an infinite one-way $k$-locally-finite graph $G$ and an $2$-coloring of the edges of $K_\NN$.  Let $\SU_1$ be a non-trivial ultrafilter on $\NN$.  Let $c_{\SU_1}$ be the vertex-coloring induced by $\SU_1$ and for all $i\in [2]$, let $A_{1,i}$ be the set of vertices receiving color $i$. Choose $i_1 \in [2]$ so that $A_{1,i_1}\in \SU_1$ and let $i_1' = 3 - i_1$.
If $A_{1,i_1'}$ is finite, then stop; otherwise, let $\SU_2$ be a non-trivial ultrafilter on $W_2 = A_{1,i_1'}$ and let $c_{\SU_2}$ be the vertex-coloring of $W_2$ induced by $\SU_2$. For all $i\in [2]$, let $A_{2,i}$ be the set of vertices receiving color $i$.
Choose $i_2$ so that $A_{2,i_2}\in \SU_2$ and let $i_2' = 3- i_2$. Let $W_3:= A_{2,i_2'}$ and continue in this manner until the point at which (a) $A_{t,i_t'}$ is finite for some $t$, or (b) there exists $t$ and $j\in [2]$ such that there exists a set $J\subseteq [t]$ where $|J|=k-1$ and $A_{j, i_j}\in \SU_j$ for all $j\in J$.  Note that by pigeonhole, we must have $t\leq 2k-3$ in either case.  

If we are in case (a), then by \eqref{fact:density} one of the sets $A_{1,i_1}, A_{2, i_2}, \dots, A_{t, i_t}$ has upper density at least $\frac{1}{t}$, say $A_{j,i_j}$.  Now by applying \cref{prop:embed0rul} with color $i_j$, we get the desired monochromatic copy of $G$ covering $A_{j,i_j}$ with upper density at least $\frac{1}{t}\geq \frac{1}{2k-3}>\frac{1}{2(k-1)}$.  If we are in case (b), suppose without loss of generality that $j=1$.  Set $W_{t+1}:=W_t\setminus A_{t,1}$.  Since $\{A_{1,1}, A_{2,1}, \dots, A_{t, 1}, W_{t+1}\}$ is a partition of $\NN$, we have by \eqref{fact:density} that one of the sets $A_{1,1}, A_{2,1}, \dots, A_{t, 1}, W_{t+1}$ has upper density at least $\frac{1}{2(k-1)}$.
If, say, $\ud(A_{\ell, i_\ell})\geq \frac{1}{2(k-1)}$ for some $\ell\in [t]$, then applying Proposition \ref{prop:embed0rul} with color $i_\ell$ gives the desired monochromatic copy of $G$ covering $A_{\ell}$;
otherwise $\ud(W_{t+1})\geq \frac{1}{2(k-1)}$, and applying Proposition \ref{prop:embed} with color $2$ gives the desired monochromatic copy of $G$ covering $W_{t+1}$.

\begin{figure}[ht]
\begin{center}
\includegraphics[height=3.5cm]{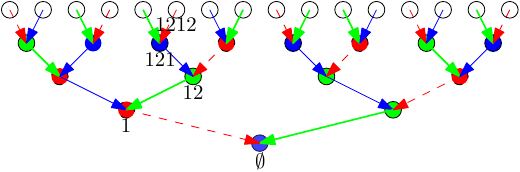}
\end{center}
\caption{An example of the proof of \cref{thm:locally-finite}.(iii) with $r=3$ and $k=3$. Here we have highlighted the sequence $A_{(1,2,1,2)},A_{(1,2,1)},A_{(1,2)},A_{(1)},A_{\emptyset}$ and note that some color, in this case red, must appear at least twice, which means we can embed $G$ into $A_{(1,2,1,2)}\cup A_{(1,2,1)}\cup A_{\emptyset}$ in such a way that $A_{(1,2,1,2)}$ is covered.}\label{fig:iii}
\end{figure}

(iii) The process is very similar to (ii), in that we repeatedly choose ultrafilters until the leftover vertices are finite, or we are guaranteed that some color appears $k-1$ times from every set at the end of the process (see \cref{fig:iii}).  However, the formal proof is a bit more technical.

We will use the following notation. Given $i_1, i_2 \in \NN$, and $L_1 \in \NN^{i_1}$ and $L_2 \in \NN^{i_2}$, we write $L_1 \prec L_2$ if $L_1$ is an initial segment of $L_2$. Furthermore, given $L = (j_1, \ldots, j_i) \in \NN^i$ for some $i \in \NN$, we define $L^- := (j_1, \ldots, j_{i-1})$.

Suppose the edges of $K_\NN$ are colored with $r$ colors and let $q = (k-2)r+1$. We will define sets $A_L$ for $L \in \bigcup_{i=0}^{q} [r-1]^{i}$ and colorings $\chi_1 : \{A_L: L \in \bigcup_{i=0}^{q-1} [r-1]^i\} \to [r]$ and $\chi_2 : \bigcup_{i=1}^{q} [r-1]^i \to [r]$ with the following properties.

\begin{enumerate}
    \item[(a)] The sets $A_L$, $L \in \bigcup_{i=0}^q [r-1]^i$, are pairwise disjoint and their union is cofinite.
    \item[(b)] For every $L \in \bigcup_{i=1}^{q} [r-1]^i$, $A_L$ is empty or every finite set $S \subset A_L$ has infinitely many common neighbors of color $\chi_1(A_L)$ in $A_L$.
    \item[(c)] For every $L \in \bigcup_{i=1}^{q} [r-1]^i$, $A_L$ is empty or every finite set $S \subset \bigcup_{L \prec L'} A_{L'}$ has infinitely many common neighbors of color $\chi_2(L)$ in $A_{L^-}$.
\end{enumerate}

We will construct these sets and colorings recursively. In the process, we will also construct sets $B_L$ and ultrafilters $\SU_{L}$ on $B_L$ for every $L \in \bigcup_{i=0}^q [r-1]^i$.

Let $B_{()} = \NN$ and let $\SU_{()}$ be a non-trivial ultrafilter on $B_{()}$, where $()$ denotes the empty sequence.  Let $c_{\SU_{()}}$ be the vertex-coloring induced by $\SU_{()}$. Let $c$ be the color so that $A_{()}$, the set of vertices of color $c$, is in $\SU_{()}$ and let $\chi_1(A_{()}) = c$. Let $[r] \setminus \{c\} = \{j_1, \ldots, j_{r-1}\}$ and, for $i \in [r-1]$, let $B_{(i)}$ be the set of vertices receiving color $j_i$ and let let $\chi_2((i)) = j_i$.

In the next step, we proceed as follows for every $i_0 \in [r-1]$. If $B_{(i_0)}$ is finite, let $A_{(i_0)} = B_{(i_0,i)} = \emptyset$ for every $i \in [r-1]$. Otherwise, let $\SU_{(i_0)}$ be a non-trivial ultrafilter on $B_{(i_0)}$ and let $c_{\SU_{(i_0)}}$ be the vertex-coloring induced by $\SU_{(i_0)}$. Let $c$ be the color so that $A_{(i_0)}$, the set of vertices of color $c$, is in $\SU_{(i_0)}$ and let $\chi_1(A_{(i_0)}) = c$. Let $[r] \setminus \{c\} = \{j_1, \ldots, j_{r-1}\}$ and, for $i \in [r-1]$, let $B_{(i_0,i)}$ be the set of vertices receiving color $j_i$ and let let $\chi_2((i_0,i)) = j_i$.

We proceed like this until we defined the sets $B_L$ for every $L \in [r-1]^q$ and let $A_L := B_L$ for all $L \in [r-1]^q$. It is easy to see from the ultrafilter properties that the above properties hold.

Therefore, for every $L \in \bigcup_{i=0}^{q-1} [r-1]^i$, $A_L$ is empty or can be covered by a monochromatic copy of $G$ by \cref{prop:embed0rul}.
Furthermore, for every $L \in [r-1]^q$ for which $A_L$ is non-empty, we find $k-1$ sets $L_1 \prec \ldots \prec L_{k-1} \prec L$ of the same color w.r.t.\ $\chi_2$ by the pigeonhole principle. Therefore, applying \cref{prop:embed} to $U_k := A_{L_1^-}, \ldots, U_2 := A_{L_{k-1}^-}, U_1:= A_L$, we find a monochromatic copy of $G$ covering $A_L$.
Since, there are $C := \sum_{i=0}^{q} (r-1)^i$ sets $A_L$, one of them has upper density at least $1/C$.
\end{proof}

Let $G$ be a graph with $\Delta:=\Delta(G)<\infty$.  Since $\chi(G)\leq \Delta(G)+1$, we immediately obtain as a corollary that $\Rd(G)\geq \frac{1}{2\Delta}$.  However, with a bit more work we obtain the following corollary.

\begin{corollary}\label{cor:maxdeg}
Let $G$ be an infinite graph.  If $2\leq \Delta:=\Delta(G)<\infty$, then $\Rd(G)\geq \frac{1}{2(\Delta-1)}$.
\end{corollary}

First we note the following fact (this result also appears in \cite[Theorem 1(i)]{L}).

\begin{proposition}\label{prop:infcomp}
Let $r\in \N$. If $G$ is a graph with infinitely many components, then $\Rd_r(G)\geq 1/r$.
\end{proposition}

\begin{proof}
Let $G$ be a graph with infinitely many components and note that by merging components if necessary, we may assume that $G$ has infinitely many components each of which has infinitely many vertices.

By Ramsey's theorem, it is possible to partition any $r$-colored $K_{\NN}$ into monochromatic infinite cliques and a finite set.  Indeed, greedily take disjoint monochromatic copies of $K_\N$ in which the smallest vertex is minimal.  Either the process ends with a finite set of uncovered vertices, or the process continues for infinitely many steps and the union misses infinitely many vertices.  However, now there is a monochromatic copy of $K_\N$ whose minimal vertex must be smaller than one of the monochromatic cliques in our collection, a contradiction.

Without loss of generality, suppose the cliques of color 1 have upper density at least $1/r$.  Since $G$ has infinitely many components, $G$ can be surjectively embedded into the cliques of color 1.
\end{proof}

\begin{proof}[Proof of \cref{cor:maxdeg}]
First note that if $G$ has infinitely many components, then we are done by \cref{prop:infcomp}.  If $\chi(G)\leq \Delta$, then we are done by \cref{thm:locally-finite}; so suppose that $G$ has finitely many components and $\chi(G)=\Delta+1$.
Now by Brooks theorem, either $\Delta=2$ and $G$ contains finitely many components which are odd cycles, or $\Delta\geq 3$ and $G$ contains finitely many components which are cliques on $\Delta+1$ vertices.  Note that in either case, every infinite component of $G$ (of which there is at least one), has chromatic number at most $\Delta$. Let $V_2 \subseteq V(G)$ be the vertex-set of the finitely many components which are odd cycles or cliques of size $\Delta+1$, and let $V_1 = V(G) \setminus V_2$.

We are given a 2-coloring of $K_\NN$.  If there is a red clique $R$ and a blue clique $B$ each of size $|V_2|$, we can apply \cref{thm:locally-finite} to $G[V_1]$ (which is one-way $\Delta$-locally finite) and $K_\N[(R \cup B)^c]$ to get a monochromatic copy of $G[V_1]$ of upper density at least $\tfrac{1}{2 (\Delta - 1)}$. Together with either $R$ or $B$, this gives the desired copy of $G$.

So suppose that there is no, say, red clique of order $|V_2|$.  If $\Delta\geq 3$, we repeat the proof of \cref{thm:locally-finite}(ii); however, in each iteration $i_j=1$ (here, blue is $1$ and red is $2$), otherwise there would be an infinite red clique. Thus we can stop when $t = \chi(G)-1 \leq \Delta$ and get a monochromatic copy of $G$ of upper density at least $\tfrac{1}{\Delta+1} \geq \tfrac{1}{2(\Delta - 1)}$.  Finally, if $\Delta=2$, we repeat the proof of \cref{thm:locally-finite}(ii), but after the first step, we have $A_{1,1}\in \SU_1$ and $W_2=A_{1,2}$.  If $\ud(A_{1,1})\geq 1/2$, then we are done as usual.  So suppose $\ud(A_{1,2})\geq 1/2$.  If there is an infinite red matching in $A_{1,2}$, then these edges can be used to make the odd cycles comprising $V_2$ and then $V_1$ can be embedded as usual.  Otherwise $A_{1,2}$ does not contain an infinite red matching and thus there is a cofinite subset of $A_{1,2}$ which induces a blue clique into which we can embed $G$.
\end{proof}

Finally we note the following strengthening of \cref{thm:locally-finite} which generalizes a result of Elekes, D. Soukup, L. Soukup, and Szentmiklóssy~\cite{ESSS} who proved a similar statement for powers of cycles.

\begin{theorem}\label{thm:locally-finite-partition}
Let $k,r\in \N$ and let $G$ be a one-way $k$-locally finite graph. In every $r$-coloring of the edges of $K_{\NN}$, there exists a collection of
\[
f(r,k)=
\begin{cases}
r & \text{ if } k=2\\
\sum_{i=0}^{(k-2)r+1} (r-1)^i & \text{ if } k\geq 3
\end{cases}
\]
vertex-disjoint, monochromatic copies of $G$ whose union covers all but finitely many vertices.
\end{theorem}

\begin{remark}
The proof of \cref{thm:locally-finite} immediately shows that for every one-way $k$-locally finite graph $G$ and every $r$-colored $K_\N$, there is a collection of at most $f(r,k)$ monochromatic copies of $G$ \emph{covering} a cofinite subset of $\N$, where $f(r,k)$ is as in the statement of \cref{thm:locally-finite-partition}.
In order to obtain a \emph{partition} as required by \cref{thm:locally-finite-partition}, we need to guarantee that these copies can be chosen to be disjoint.  To do so, instead of applying \cref{prop:embed,prop:embed0rul}, we will embed the graphs simultaneously doing one step of the embedding algorithms of \cref{prop:embed,prop:embed0rul} at a time always making sure not to repeat vertices (which is possible since we have infinitely many choices in every step but only finitely many embedded vertices). Otherwise, the proof is exactly the same and therefore we will omit it.
\end{remark}

\section{Graphs of bounded ruling number}\label{sec:ruling}

In this section, we will prove \cref{thm:ruling}.

\begin{proof}[Proof of \cref{thm:ruling}]
  Let $G$ be a finitely ruled graph and suppose $K_\N$ is colored with $r$-colors for some $r\in \NN$. Let $\SU$ be a positive ultrafilter on $\N$ and denote by $V_i$ the set of vertices of color $i$ in the vertex-coloring induced by $\SU$. Suppose without loss of generality that $V_1 \in \SU$.
  Since $G$ is finitely ruled, there is a finite set $S\subseteq V(G)$ such that $G[S^c]$ does not have any finite dominating set and in particular $G[S^c]$ is $0$-ruled.

  We will now construct the embedding $f: V(G) \to \N$.
  First embed $S$ into an arbitrary clique of color $1$ in $V_1$ of size $|S|$ (such a clique can be found be iteratively applying the ultrafilter property). Let $V_1^0=\Ncap_1(f(S))\cap V_1$ and note that $V_1^0 \in \SU$ and hence satisfies the assumptions of \cref{prop:embed0rul}. Therefore, $G[S^c]$ can be surjectively embedded into $V_1^0$, and we can extend $f$ to an embedding of $G$. Since $V_1^0 \subset f(V(G))$ has positive upper density, we are done.
\end{proof}
%%%%%%%%%%%%%%%%%%%%%%%%%%%%%%%%%%%%%%%%%%%%%%%%%%%%%%%%%%%%%%%%%%%%%%%%%
%%%%%%%%%%%%%%%%%%%%%%%%%%%%%%%%%%%%%%%%%%%%%%%%%%%%%%%%%%%%%%%%%%%%%%%%%
%%%%%%%%%%%%%%%%%%%%%%%%%%%%%%%%%%%%%%%%%%%%%%%%%%%%%%%%%%%%%%%%%%%%%%%%%
\section{Graphs of bounded degeneracy}\label{sec:degen}
Given $k\in \N$ and a graph $G$, we say that $X\subseteq V(G)$ is \emph{$k$-wise intersecting} if for all $S\subseteq X$ with $|S|\leq k$, $N^{\cap}(S)$ is infinite.  We say that $X\subseteq V(G)$ is \emph{$k$-wise self-intersecting} if for all $S\subseteq X$ with $|S|\leq k$, $X\cap N^{\cap}(S)$ is infinite.
We say that a graph $G$ is \emph{$k$-wise intersecting} if $V(G)$ is $k$-wise intersecting (and consequently $k$-wise self-intersecting).  Finally, if $G$ is an $r$-colored graph for some $r\in \N$, we say that $X\subseteq V(G)$ is \emph{$k$-wise (self-)intersecting in color $i$} if $X$ is $k$-wise (self-)intersecting in $G_i$.

The following is related to Proposition \ref{prop:embed0rul}.

\begin{proposition}\label{prop:embed0ruldeg}
Let $d\in \N$ and let $G$ be an infinite, 0-ruled, $d$-degenerate graph.  If $H$ is a $(d+1)$-wise intersecting graph, then we can surjectively embed $G$ into $H$.
\end{proposition}

\begin{proof}
Do the same as in the proof of \cref{prop:embed0rul}, but since $G$ is $d$-degenerate, when we get to the second phase of the embedding step, where we embed all vertices from $\{v_{n+1}, \dots, v_{p-1}\}$ into $H$ one at a time, we note that each vertex $v_i$ is adjacent to at most $d+1$ vertices in $\{v_1, \dots, v_{i-1}\}\cup \{v_p\}$, so it is possible to choose an image for $v_i$ in $H$.
\end{proof}

In the proofs of \cref{thm:locally-finite} and \cref{thm:ruling}, we implicitly proved the following.  However, for completeness, we will give a short proof.  

\begin{proposition}
Let $r\in \N$.  For every $r$-coloring of $K_\NN$, there is a set $X$ with upper density at least $1/r$ and a color $i\in [r]$ such that for every $k\in \N$, $X$ is $k$-wise intersecting in color $i$.  Moreover there is a set $Y$ with positive upper density and a color $i\in [r]$ such that for every $k\in \N$, $Y$ is $k$-wise self-intersecting in color $i$.
\end{proposition}

\begin{proof}
Let $\CU$ be a positive ultrafilter on $\N$ (that is, $\CU$ is an ultrafilter such that every set in $\CU$ has positive upper density).  For all $i\in [r]$ let $V_i$ be the set of vertices $v\in \N$ such that $N_i(v)\in \CU$.  Let $X$ be the set in $\{V_1, \dots, V_r\}$ with the largest upper density and let $Y$ be the set in $\{V_1, \dots, V_r\}$ which is in $\CU$.  
\end{proof}

Note that by Proposition \ref{prop:embed0ruldeg}, for the purposes of embedding $0$-ruled, $d$-degenerate graphs we don't need the set $Y$ described above to be $k$-wise self-intersecting for all $k\in \N$.  Thus we can ask if it is possible to find a set $Y$ which is $(d+1)$-wise self-intersecting and has upper density bounded below by some function of $d$.  While we haven't been able to address this question, we now give an example which provides an upper bound on the upper density of such a set.  This example is due to Chris Lambie-Hanson \cite{CLH}.

\begin{proposition}\label{prop:selfint}
For all $k\in \N$, there exists a $2$-coloring of $K_\NN$ such that every monochromatic $k$-wise
self-intersecting set has upper density at most $1/2k$.
\end{proposition}
\begin{proof}
Let $k\in \NN$ and partition $\NN$ into sets $A_1,\ldots,A_k$ and $B_1,\ldots,B_k$ of
equal asymptotic density $1/2k$.  Let $A = A_1\cup\cdots\cup A_k$ and $B =
B_1\cup \cdots \cup B_k$.  The coloring is as follows.  Given $a\in A$ and
$b\in B$, we color $\{a,b\}$ red if $a < b$ and blue otherwise.  Given
$a,a'\in A$, we color $\{a,a'\}$ red if $a$ and $a'$ are in the same set
$A_i$, and blue otherwise.  Given $b,b'\in B$, we color $\{b,b'\}$ blue if
$b,b'$ are in the same set $B_i$, and red otherwise.

\begin{figure}[ht]
\begin{center}
	\includegraphics[height=5cm]{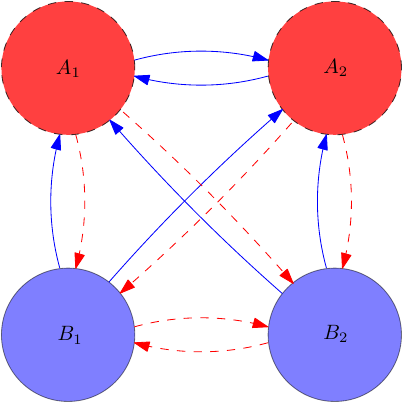}
\end{center}
\caption{An example of the coloring from \cref{prop:selfint} in the case when $k=2$.  The shaded areas denote cliques of the respective colors and a blue/solid (red/dashed) arrow from one part to another indicates that vertices in the first part have cofinitely many blue (red) neighbors in the second part.}
\label{fig:degen}
\end{figure}

The colors are clearly symmetric so it suffices to consider a red $k$-wise
self-intersecting set $X$.  We claim that $X$ is contained in a single
$A_i$.

Note that for all $b\in B_i$, $N_R(b)\cap A$ is finite and $N_R(b)\cap B_i =
\emptyset$.  Thus if $X\cap B_i\neq\emptyset$, $X\cap B_j\neq \emptyset$ for
some $j\neq i$.  Applying the same argument with elements of $B_i$ and
$B_j$, we see that $X\cap B_h\neq\emptyset$ for some $h\neq i,j$, and
continuing we get $X\cap B_\ell\neq\emptyset$ for all $\ell = 1,\ldots,k$.
But then taking $F$ to be a subset of $X$ consisting of one vertex from each
$B_\ell$, we see that $N^{\cap}_R(F)$ is finite, a contradiction.

So we must have $X\subseteq A$.  But note that $N_R(a)\cap A\subseteq A_i$
for all $a\in A_i$.  Hence $X$ must be contained in $A_i$ for some $i$.
\end{proof}

It is not immediately clear that there exists a $d$-wise intersecting graph with bounded degeneracy.  So we now give a construction of a family of $d$-wise intersecting graphs which are $d$-degenerate (and 0-ruled).

\begin{proposition}\label{prop:constr}
  For every $d \in \NN$, there is an infinite graph $H_d$ which is $d$-wise intersecting, $d$-degenerate, and 0-ruled.
\end{proposition}

\begin{proof}%[Proof of \cref{prop:constr}]
Let $n_0=d$.  For all $i\geq 0$, let $S_1, S_2, \dots, S_{\binom{n_{i}}{d}}$ be an enumeration of all the $d$-element subsets of $[n_{i}]$, let $n_{i+1}=n_i+\binom{n_i}{d}$, and let
\[E_{i+1}=\bigcup_{1\leq j\leq \binom{n_{i}}{d}}\{\{n_i+j, v\}: v\in S_j\}.\]
Let $H_d$ be the graph on vertex set $\NN$ with edge set $\bigcup_{j \in \NN} E_j$.

By the construction it is clear that $H_d$ is $d$-wise intersecting and $d$-degenerate.  To see that $G$ is 0-ruled, note that for any finite set $X\subseteq \NN$ and any $d$-element set $Y\subseteq \NN\setminus X$, there are infinitely many vertices which are adjacent to every vertex in $Y$ and none of the vertices in $X$. Thus $G$ cannot have a finite dominating set.
\end{proof}

Note that, in particular, $H_d$ contains a spanning copy of every $(d-1)$-degenerate $0$-ruled graph (by Proposition \ref{prop:embed0ruldeg}). Denote by $\rho(d)$ the smallest Ramsey upper density of a $d$-degenerate infinite graph and by $\tau(d)$ the largest $\tau \geq 0$ such that every $2$-colored complete graph contains a monochromatic $d$-wise self-intersecting subgraph of density at least $d$. The above propositions imply
\[ \tau(d-1) \geq \rho(d-1) \geq \tau(d) \geq \rho(d)\]
for every $d \geq 2$. In particular, we have $\tau(d)>0$ for every $d \in \NN$ if and only if $\rho(d) > 0$ for every $d \in \NN$.  So in order to answer \cref{qu:degenerate} positively for 0-ruled graphs, it would suffice to answer \cref{qu:degenerate} positively for $H_d$ for all $d$.  Note that $H_1=T_\infty$ and thus \cref{thm:trees} gives a positive answer for the case $d=1$.

We conclude this section with a few comments about \cref{qu:degenerate}.

In light of Theorem \ref{thm:locally-finite}, if there was a function $f:\NN\to \NN$ such that for all $d\in \NN$, every $d$-degenerate graph is one-way $f(d)$-locally finite, then we would have a positive answer to \cref{qu:degenerate}; however, this is not the case as there are $d$-degenerate graphs which are not one-way $k$-locally-finite for any $k$.  For instance, the graph $H_d$ constructed above is $d$-degenerate, but since every vertex has infinite degree, is not one-way $k$-locally-finite for any $k$.  Also $K_{d,\NN}$ is $d$-degenerate but not one-way $k$-locally-finite for any $k$ (although in this case, we know $\Rd(K_{d,\NN})\geq \frac{1}{2^{2d-1}}$).

\cref{qu:degenerate} is about all $d$-degenerate graphs.  However, the discussion in this section is about $0$-ruled, $d$-degenerate graphs.  It seems possible that answering \cref{qu:degenerate} positively for $0$-ruled, $d$-degenerate graphs could imply a positive answer for all $d$-degenerate graphs (c.f.\ the proof of \cref{thm:ruling}).

\begin{problem}\label{q:ruldeg}
If \cref{qu:degenerate} were true for all $0$-ruled $d$-degenerate graphs (in particular,  $H_d$), would this imply that \cref{qu:degenerate} was true for all $d$-degenerate graphs?
\end{problem}

Let $d\geq 2$ and say that a digraph $D$ is \emph{$d$-directed} if every $d$-set in $V(D)$ has a common out-neighbor; that is, for all $S\subseteq V(D)$ with $|S|= d$ there exists $w\in V(D)$ (where it is possible for $w\in S$) such that for all $v\in S$, $(v,w)\in E(D)$.  For example, the digraph $D=(\{a,b,c\}, \{(a,b), (b,c), (c,a), (a,a), (b,b), (c,c)\})$ is 2-directed, but not 3-directed.

In order to get a monochromatic $d$-wise self-intersecting set with upper density at least some fixed amount in an arbitrary 2-coloring of $K_\NN$, we likely have to solve the following problem\footnote{Since we first posted this paper, this problem has essentially been resolved \cite{DG} (although, determining the minimum number of $d$-directed graphs needed to cover $V(K)$ is still open).}.

\begin{problem}\label{prob:2dir}
Given a 2-coloring of the edges of a complete (finite) digraph $K$ (including loops), is it possible to cover $V(K)$ with at most $2d$ monochromatic $d$-directed graphs?  (If not $2d$, some other bound depending only on $d$?)
\end{problem}

The reason is that given any 2-coloring of a complete digraph $K$ (plus loops), we can create a corresponding 2-coloring of $K_{\NN}$ as follows.  Split $\NN$ into infinite sets $A_i$, one for each vertex $i$ of $K$.  Color the edges inside $A_i$ according to the color of the loop on $i$.  Now if both directed
edges $(i,j)$ and $(j,i)$ are the same color, give all edges between $A_i$
and $A_j$ that color; if not, then color the bipartite graph between $A_i$
and $A_j$ with the bipartite half graph coloring (where $(i,j)$ being red means that the vertices in $A_i$ have cofinite red degree to $A_j$).  Then any $d$-wise
self-intersecting set $B$ must be the union of some collection of $A_i$'s
whose corresponding vertices $i$ make up a monochromatic $2d$-directed
set in $K$.
%%%%%%%%%%%%%%%%%%%%%%%%%%%%%%%%%%%%%%%%%%%%%%%%%%%%%%%%%%%%%%%%%%%%%%%%%%%%
%%%%%%%%%%%%%%%%%%%%%%%%%%%%%%%%%%%%%%%%%%%%%%%%%%%%%%%%%%%%%%%%%%%%%%%%%%%%
%%%%%%%%%%%%%%%%%%%%%%%%%%%%%%%%%%%%%%%%%%%%%%%%%%%%%%%%%%%%%%%%%%%%%%%%%%%%
\section{Bipartite Ramsey densities}\label{sec:bipartite}

In this section we prove \cref{thm:bip-paths-partition}.
An infinite graph $G$ is said to be \emph{infinitely connected} if $G$ remains connected after removing any finite set of vertices.
Note that every vertex of an infinitely connected graph has infinite degree.
Given some set of vertices $S \subset V(G)$, we say that $S$ is \emph{infinitely connected} if $G[S]$ is infinitely connected.
Similarly, we call a set $S \subset V(G)$ \emph{infinitely linked} if for all distinct $u,v\in S$, there are infinitely many internally vertex-disjoint paths in $G$ from $u$ to $v$ (note that the internal vertices of these paths need not be contained in the set $S$).
Note that every infinitely connected set is also infinitely linked but the converse is not true (for example, both parts of $K_{\NN,\NN}$ are infinitely linked but not connected).
Further note that if $S_1, \dots, S_k$ are sets, each of which is infinitely linked, then there are disjoint paths $P_1, \dots, P_k$ such that $P_1\cup \dots \cup P_k$ covers $S_1\cup \dots \cup S_k$.

If $G$ is a colored graph and $c$ is a color, we say that $G$ is \emph{infinitely connected in $c$} if $G_c$ (the spanning subgraph of $G$ with all edges of color $c$) is infinitely connected.
A set $S \subseteq V(G)$ is \emph{infinitely connected in color $c$} (\emph{infinitely linked in color $c$}) if $S$ is infinitely connected (infinitely linked) when restricted to $G_c$.
$S$ is called monochromatic infinitely connected (infinitely linked) if it is infinitely connected in some color $c$.

The following proposition directly implies \cref{thm:bip-paths-partition} which implies \cref{thm:bip-paths}.

\begin{proposition}\label{thm:bip}
  Every $2$-colored $K_{\NN,\NN}$ can be partitioned into a finite set and two monochromatic infinitely linked sets $X$ and $Y$.
\end{proposition}

\begin{proof}
Let $V_1, V_2$ be the parts of the bipartite graph and let $\SU_1, \SU_2$ be non-trivial ultrafilters on $V_1$ and $V_2$. For $ i =1,2$, let $B_i \subset V_i$ be the blue vertices in the induced vertex-coloring and let $R_i = V_i \setminus B_i$ be the red vertices.

\textbf{Case 1 ($|R_1|=|R_2|=|B_1|=|B_2| = \infty$).} If there are infinitely many disjoint red paths between $R_1$ and $R_2$, then $X := R_1 \cup R_2$ is infinitely linked in red. Indeed, if $v_1,v_2 \in R_1$ or $v_1,v_2 \in R_2$, then they have infinitely many common red neighbors (by the properties of the ultrafilter).
If $v_1 \in R_1$ and $v_2 \in R_2$, we will construct infinitely many internally disjoint paths between $x_0:=v_1$ and $x_5:=v_2$ as follows: let $P=x_2\dots x_3$ be a red path so that $x_2 \in R_1$ and $x_3 \in R_2$, and let $x_1$ be a common red neighbor of $x_0$ and $x_2$ (of which we have infinitely many as above) and $x_4$ be a common neighbor of $x_3$ and $x_5$. It is clear that $x_0x_1x_2\dots x_3x_4x_5$ defines a red path and that we can construct infinitely many internally disjoint paths like this.
If there are only finitely many disjoint red paths between $R_1$ and $R_2$, then there is a finite set $S$ so that, in particular, $X := (R_1 \cup R_2) \setminus S$ induces a complete blue bipartite graph with parts of infinite size and hence is infinitely linked in blue.
Similarly, there is a set $Y \subset B_1 \cup B_2$ which is cofinite in $B_1 \cup B_2$ and infinitely linked in red or infinitely linked in blue.

\textbf{Case 2.} Suppose without loss of generality that $R_1$ is finite. It is easy to verify that $X = B_1 \cup B_2$ is infinitely linked in blue and $Y := R_2$ is infinitely linked in red.
\end{proof}

As mentioned in the introduction, the above result combined with a recent result of Day and Lo \cite{DL} implies that $\Rd_3(P_\infty)=\frac{1}{2}$.

\section{Trees}\label{sec:trees}
%%%%%%%%%%%%%%%%%%%%%%%%%%%%%%%%%%%%%%%%%%%%%%%%%%%%%%%%%%%%%%%%%%%%%%%
\subsection{General embedding results}
Given $k\in \N$, we say that a connected graph $T$ has \emph{radius} at most $k$ if there exists $u\in V(T)$ such that for all $v\in V(T)$, there is a path of length at most $k$ from $u$ to $v$; if no such $k$ exists we say that $T$ has \emph{unbounded radius}.

\begin{lemma}\label{lem:embed-deep-tree}
Let $T$ be a graph.  A spanning copy of $T$ can be found in every infinitely connected graph $H$ if and only if $T$ is a forest and (i) $T$ has a component of unbounded radius or (ii) $T$ has infinitely many components.
\end{lemma}

In order to simplify the proof of Lemma \ref{lem:embed-deep-tree}, we first prove the following structural result about trees with unbounded radius.  An \emph{increasing star} is a tree obtained by taking an infinite collection of disjoint finite paths of unbounded length and joining one endpoint of each of the paths to a new vertex $v$.  Note that an increasing star has unbounded radius, no infinite path, and exactly one vertex of infinite degree (which is called the \emph{center}).  Also note that an increasing star has distinct vertices $v_0, v_1, v_2, \dots$ and internally disjoint paths $P_1, P_2, \dots$ such that for all $i\geq 1$, $P_i$ is a path from $v_0$ to $v_i$ and the length of $P_{i+1}$ is greater than the length of $P_i$.

\begin{fact}\label{fact:unbound-radius}
Let $T$ be a tree of unbounded radius.  Either for all $v\in V(T)$, there is an infinite path in $T$ starting with $v$ or there exists $v_0\in V(T)$ such that $T$ contains an increasing star having $v_0$ as the center.
\end{fact}

\begin{proof}
Let $T$ be a tree, let $v\in V(T)$, and suppose there is no infinite path in $T$ starting with $v$ (since $T$ is connected, this implies that there is no infinite path in $T$ at all).  Since $T$ has unbounded radius, we can do the following: let $Q_1$ be a path from $v$ to a leaf $u_1$, which has some length $k_1$, now there must exist a path $Q_2$ of length $k_2>k_1$ from $v$ to a leaf $u_2$, and so on.  This process gives an infinite set of leaves $U$ and an increasing sequence $k_1, k_2, \dots$ such that there is a path from $v$ to $u_i$ of length $k_i$.  Now we apply the Star-Comb lemma \cite[Lemma 8.2]{Dies} to the set $U$.  Since $T$ has no infinite path, there must exist a subdivision of an infinite star with center $v_0$ such that all the leaves, call them $U'$, are in $U$.  We claim that for all $k$ there exists a path from $v_0$ to $U'$ which has length greater than $k$, which would prove the lemma.  If not, then there exists $k$ such that every path from $v_0$ to $U'$ has length at most $k$.  However, since there is a path from $v$ to $v_0$, this would imply that there exists a $k'$ such that every path from $v$ to $U'$ has length at most $k'$.  But this contradicts the fact that the lengths of the paths from $v$ to $U'$ form an increasing sequence.
\end{proof}

\begin{proof}[Proof of \cref{lem:embed-deep-tree}]
First suppose that a spanning copy of $T$ can be found in every infinitely connected graph $H$.  It is known that there exist infinitely connected graphs with arbitrarily high girth (see \cite[Chapter 8, Exercise 7]{Dies}); for instance, let $H_0$ be a cycle of length $k$, and for all $i\geq 1$, let $H_i$ be the graph obtained by adding a vertex $x_i$ and internally disjoint paths of length $\ceiling{k/2}$ from $x_i$ to every vertex in $H_{i-1}$, then let $H=\cup_{i\geq 0}H_i$.  So $H$ is infinitely connected and has girth $k$. This proves that $T$ is acyclic (i.e. $T$ is a forest) because otherwise there would be an infinitely connected graph in which every cycle is longer than the shortest cycle in $T$.

Let $H$ be the infinite blow-up of a one-way infinite path (i.e.\ replace each vertex with an infinite independent set and each edge with a complete bipartite graph). Clearly $H$ is infinitely connected.  If $T$ is a spanning subgraph of $H$, then $T$ has a component of unbounded radius or $T$ has infinitely many components.  

Thus $T$ must be a forest with a component of unbounded radius or infinitely many components.

Next suppose $T$ is a forest with a component of unbounded radius or infinitely many components.  If $T$ has infinitely many components $T_1, T_2, \dots$, we may select for all $i\geq 1$, $t_i\in V(T_i)$ and add the edge $t_it_{i+1}$ for all $i\geq 1$ to get a tree with unbounded radius which contains $T$ as a spanning subgraph.  If $T$ has finitely many components $T_1, \dots, T_k$, at least one of which has unbounded radius, we may for all $i\in [k-1]$ add an edge from $t_{i}\in V(T_i)$ to $t_{i+1}\in V(T_{i+1})$ to get a tree with unbounded radius which contains $T$ as a spanning subgraph.  Thus it suffices to prove the result when $T$ is a tree with unbounded radius.

  By \cref{fact:unbound-radius}, there exists a vertex $t_0$ such that either there is an infinite path starting with $t_0$ (in which case we say $T$ is of \emph{Type 1}), or an increasing star having $t_0$ as the center (in which case we say $T$ is of \emph{Type 2}).  Now starting with $t_0$, fix an enumeration of $V(T) = \{t_0, t_1, t_2 \ldots\}$ such that for all $i\geq 1$, $T[\{t_0, \dots, t_i\}]$ is connected (in fact, for all $i\geq 1$, $t_i$ has exactly one neighbor in $\{t_0, \dots, t_{i-1}\}$).  Also fix  an enumeration of $V(H) = \{v_0, v_1, v_2, \ldots\}$.   We will build an embedding $f$ of $T$ into $H$ recursively, in finite pieces, at each stage ensuring that we add the first vertices of $V(T) \sm \dom{f}$ and $V(H) \sm \ran{f}$ into the domain and range of $f$ respectively.

Initially, let $f(t_0)=v_0$ (we think of $t_0$ as being the root of the tree and $v_0$ as the embedding of the root in $H$) and let $t_{last}:=t_0$ and $v_{last}:=v_0$.  We now show that \cref{alg:longtree} gives the desired embedding.

\begin{algorithm}[ht]
\caption{}
\label{alg:longtree}
\begin{algorithmic}[1]
  \While{True}
   \If{$V(H)\setminus \ran f\neq \emptyset$}
      \State Let $next$ be the smallest index such that $v_{next}\in V(H) \setminus \ran f$.
      \State Let $P_{next} \subseteq H$ be a finite path from $v_{next}$ to $v_{last}$ which is internally disjoint from $\ran f$.\label{alg1}
      \State Let $V_{next}$ be a set of $|V(P_{next})| - 1$ vertices in $V(T)\setminus \dom f$ such that $\{t_{last}\}\cup V_{next}$ induces a path in $T$.\label{alg2}
      \State Extend $f$ by embedding $V_{next}$ into $V(P_{next}) \sm \{v_{last}\}$.
      \If{$T$ is of Type 1}
        \State Set $t_{last}:=f^{-1}(v_{next})$ and $v_{last}:=v_{next}$.
      \EndIf
    \EndIf
    \If{$V(T)\setminus \dom f\neq \emptyset$}
      \State Let $next$ be the smallest index such that $t_{next}\in V(T)\setminus \dom f$.
      \State Let $back<next$ be the unique index such that $t_{back}$ is adjacent to $t_{next}$.\label{alg3}
      \State Embed $t_{next}$ into an arbitrary vertex in $N_H(f(t_{back})) \setminus \ran f$.\label{alg4}
      \If{$T$ is of Type 1 and $t_{back}=t_{last}$}
        \State Set $t_{last}:=t_{next}$
      \EndIf
    \EndIf
   \EndWhile
\end{algorithmic}
\end{algorithm}
Note that if $T$ is of Type 2, then $t_{last}=t_0$ and $v_{last}=v_0$ throughout the process.

To see that $f$ is a well-defined surjective embedding of $T$ into $H$, first note that we can always follow lines \ref{alg1} and \ref{alg4} of \cref{alg:longtree} since $H$ is infinitely connected and in particular every vertex has infinite degree.  Line \ref{alg2} is always possible since there is either an infinite path starting at $v_0$ or an increasing star having $v_0$ as the center.  Line \ref{alg3} is always possible by the enumeration of $V(T)$.  So $f$ is well-defined.

We alternate between embedding the vertex $t$ of smallest index from $T$ which has not yet been embedded into an available vertex from $H$ in such that way that the parent $t'$ of $t$ has already been embedded and $f(t)$ is adjacent to $f(t')$, and embedding a path $t_0t_1\dots t_\ell$ to a vertex such that $f(t_0)f(t_1)\dots f(t_\ell)$ is a path in $H$ and $f(t_\ell)$ is the vertex of smallest index from $V(H)$ which has yet to be mapped to.  So $f$ will be a surjective embedding of $T$.
\end{proof}

Now we prove another useful lemma.

\begin{lemma}\label{lem:embed-short-tree}
Let $T$ be a tree with at least one vertex of infinite degree. If $H$ is a graph in which every vertex has infinite degree, then for all $v \in V(H)$, $H$ contains a copy of $T$ covering $N_H(v)$.
\end{lemma}

\begin{proof}
Let $t_1 \in V(T)$ be a vertex of infinite degree and let $v_1=v$ from the statement of the theorem (again we think of $t_1$ as being the root of the tree and $v_1$ as the embedding of the root in $H$).
We will build an embedding $f$ of $T$ into $H$ recursively, in finite pieces, at each stage adding one more child of every previously embedded $t \in T$ (unless all children have been embedded already). The embedding strategy is very similar to that in the proof of \cref{lem:embed-deep-tree}.
Initially, let $f(t_1) = v_1$. We will use the following \cref{alg:short-tree}.

\begin{algorithm}[ht]
\caption{}
\label{alg:short-tree}
\begin{algorithmic}[1]
  \While{True}
    \For{$t \in \dom f$}
      \If{$S := N_T(t) \sm \dom f$ is non-empty}
      \State Embed $\min(S)$ into $\min(N_H(f(t)) \setminus \ran f)$.
      \EndIf
    \EndFor
  \EndWhile
\end{algorithmic}
\end{algorithm}

First, note that we can always follow line 4 of \cref{alg:longtree} since every vertex in $H$ has infinite degree. Let $f :V(T) \to V(H)$ be the function produced by \cref{alg:short-tree}. We need to prove that $f$ is well-defined, an embedding of $T$ and that $N_H(v) \subset \dom f$.

Since we always embed the smallest not yet embedded neighbor of every previously embedded $t \in V(T)$ in line 4, every other vertex will be embedded eventually as well. Therefore, $f$ is well-defined.
Furthermore, by construction of $f$, it defines a proper embedding (whenever a new vertex $t \in T$ is embedded, its parent $t'$ is already embedded and we make sure that $f(t)$ is adjacent to $f(t')$).
Finally note that we are infinitely often in line 4 when $t = t_1$ since $N_T(t_1)$ is infinite. Since we always choose the smallest available vertex in $N_H(v) \setminus \ran f$, it follows that $ N_H(v) \subset \ran f$.
\end{proof}
%%%%%%%%%%%%%%%%%%%%%%%%%%%%%%%%%%%%%%%%%%%%%%%%%%%%%%%%%%%%%%%%%%%%%%%
%%%%%%%%%%%%%%%%%%%%%%%%%%%%%%%%%%%%%%%%%%%%%%%%%%%%%%%%%%%%%%%%%%%%%%%
\subsection{Upper density of monochromatic trees}
In this section we will deduce \cref{thm:trees} from Lemma \ref{lem:embed-deep-tree}, Lemma \ref{lem:embed-short-tree}, and the following two lemmas.

\begin{lemma}\label{lem:find-deep-tree}
  For any $2$-coloring of $K_\NN$, there are sets $R$ and $S$ such that
  \begin{enumerate}
    \item  $R \cup S$ is cofinite,
    \item  if $R$ is infinite, then it is infinitely connected in red, and
    \item  if $S$ is infinite, then it is infinitely connected in one of the colors.
  \end{enumerate}
\end{lemma}

\begin{lemma}
  \label{lem:find-short-tree}
  Let $H$ be a $2$-colored $K_{\NN}$. There exists a set $A \subset \NN$, a vertex $v \in A$, and a color $c$ such that every vertex in $F := H_c[A]$ has infinite degree and $\ud(N_F(v)) \geq 1/2$.
\end{lemma}

It is now easy to prove \cref{thm:trees}.

\begin{proof}[Proof of \cref{thm:trees}]
It clearly suffices to prove the result for trees, so let $T$ be an infinite tree and suppose the edges of $K_\N$ are colored with two colors.  If $T$ does not have an infinite path, it must have at least one vertex of infinite degree (by K\"onig's infinity lemma \cite{K}) and therefore the theorem follows immediately from \cref{lem:find-short-tree,lem:embed-short-tree}.
So suppose $T$ has an infinite path. By \cref{lem:find-deep-tree}, there is an infinite set $A$ with $\ud(A) \geq 1/2$ and a color $c$, so that the induced subgraph on $A$ is infinitely connected in $c$. By \cref{lem:embed-deep-tree}, there is a monochromatic copy of $T$ spanning $A$ and we are done.
\end{proof}

It remains to prove the two lemmas.

\begin{proof}[Proof of \cref{lem:find-deep-tree}]
Fix a $2$-coloring of $K_\NN$.  We define a sequence of sets $R_\alpha, S_\alpha$, for all ordinals $\alpha$, as follows.  Let $S_0 = \NN$. For each $\alpha$, let $R_\alpha$ be the set of vertices in $S_\alpha$ whose blue neighborhood has finite intersection with $S_\alpha$, and let $S_{\alpha+1} = S_\alpha\sm R_\alpha$.  If $\lambda$ is a limit ordinal, then we let $S_\lambda$ be the intersection of the sets $S_\alpha$, for $\alpha < \lambda$.

Note that the sets $R_\alpha$ are pairwise disjoint, and hence there is some countable ordinal $\gamma$ such that $R_\alpha = \emptyset$ for all $\alpha\ge\gamma$.  Let $\gamma^*$ be the minimal ordinal such that $R_{\gamma^*}$ is finite; it follows then that $R_\beta = \emptyset$ for all $\beta > \gamma^*$.  Set
\[R = \bigcup\set{R_\alpha}{\alpha < \gamma^*}.\]
(Note that $\gamma^*$ may be $0$, in which case $R = \emptyset$.)

Suppose that $R$ is infinite.  Then $\gamma^* > 0$ and $R_\alpha$ is infinite for all $\alpha < \gamma^*$. Let $u,v\in R$ with $u\in R_\alpha$ and $v\in R_\beta$ for some $\alpha \le \beta < \gamma^*$.
It follows that the red neighborhoods of both $u$ and $v$ are cofinite in $R_\beta$. Since $R_\beta$ is infinite, this implies that there is a red path of length $2$ connecting $u$ and $v$, even after removing a finite set of vertices.  Hence $R$ is infinitely connected in red.

Set $S = S_{\gamma^*+1}$.  Then $R\cup S = \NN\sm R_{\gamma^*}$, so $R\cup S$ is cofinite.  Moreover, since $R_{\gamma^*+1} = \emptyset$, it follows that for every $v\in S$, the blue neighborhood of $v$ has infinite intersection with $S$.  Now suppose that $S$ is not infinitely connected in blue.
Then there is a finite set $F\subseteq S$ and a partition $S\sm F = X\cup Y$ such that $X$ and $Y$ are both nonempty, and every edge between $X$ and $Y$ is red.  Note that $X$ and $Y$ must both be infinite, since if $x_0\in X$ and $y_0\in Y$ then $X\cup F$ and $Y\cup F$ must contain the blue neighborhoods of $x_0$ and $y_0$ (both of which are infinite) respectively.  But then the red graph restricted to $X\cup Y=S\setminus F$ is infinitely connected.
\end{proof}

\begin{proof}[Proof of \cref{lem:find-short-tree}]
Fix a $2$-coloring of $K_\NN$. Similarly as in the proof of \cref{lem:find-deep-tree} we will construct sets $R_\alpha, B_\alpha, S_\alpha$ for all ordinals $\alpha$.  Let $S_0 = \NN$. For each $\alpha$, let $R_\alpha$ be the set of vertices in $S_\alpha$ whose blue neighborhood has finite intersection with $S_\alpha$, let $B_{\alpha}$ be the set of vertices in $S_{\alpha}$ whose red neighborhood has finite intersection with $S_{\alpha}$, and let $S_{\alpha+1} = S_\alpha\sm (R_\alpha\cup B_{\alpha})$.  If $\lambda$ is a limit ordinal, then we let $S_\lambda$ be the intersection of the sets $S_\alpha$, for $\alpha < \lambda$.  As in the proof of \cref{lem:find-deep-tree}, the following properties hold.

\begin{enumerate}
	\item There is a unique ordinal $\gamma^{*}$ such that $R_\alpha \cup B_\alpha$ is infinite for all $\alpha < \gamma^*$, finite for $\alpha = \gamma^*$ and empty for all $\alpha > \gamma^*$. We denote $R = \bigcup_{\alpha<\gamma^*} R_\alpha$ and $B = \bigcup_{\alpha<\gamma^*} B_\alpha$.
	\item $S_\alpha = S_{\alpha'}$ for all ordinals $\alpha, \alpha' > \gamma^*$. We denote $S = S_{\gamma^* + 1}$.
	\item $R, B, S$ are pairwise disjoint and $R \cup B \cup S$ is cofinite.
	\item If $v \in R_\gamma$ for some ordinal $\gamma$, then $v$ has finitely many blue neighbors in $S \cup \bigcup_{\alpha \geq \gamma} R_\alpha$.
	\item If $v \in B_\gamma$ for some ordinal $\gamma$, then $v$ has finitely many red neighbors in $S \cup \bigcup_{\alpha \geq \gamma} B_\alpha$.
	\item Every $v \in S$ has infinitely many neighbors of both colors in $S$.
\end{enumerate}

If $R \cup B$ is empty, then let $A = S$ and choose an arbitrary vertex $v \in S$. Since $A$ is cofinite in $\N$, either the blue or the red neighborhood of $v$ in $A$ has upper density at least $1/2$. Since every vertex in $A$ has infinite degree in both colors, we are done.

If $R \cup B$ is non-empty it must be infinite (by the way $R$ and $B$ are defined). Since $R \cup B \cup S = (R \cup S) \cup (B \cup S)$ is cofinite, we may assume without loss of generality that $R$ is non-empty and $\ud(R \cup S) \geq 1/2$. Let $A = R \cup S$ and let $v \in R_0$ be arbitrary (if $R$ is non-empty, then $R_0$ must be infinite).  Clearly every vertex in $A$ has infinite red degree in $A$ and since $v$ has only finitely many blue neighbors in $A$, we are done.
\end{proof}
%%%%%%%%%%%%%%%%%%%%%%%%%%%%%%%%%%%%%%%%%%%%%%%%%%%%%%%%%%%%%%%%%%
\subsection{Ramsey-cofinite forests}
In this section we will prove \cref{thm:cofinitetree}.

%\begin{theorem}\label{thm:cofiniteforest}
%Let $T$ be a forest.
%\begin{enumerate}
%\item If $T$ is strongly contracting, has no finite dominating set, and $T \not \in \CT^{*}$, then $T$ is Ramsey-cofinite.
%\item If $T$ is weakly expanding, has a finite dominating set, or $T \in \CT^{*}$, then $T$ is not Ramsey-lower-dense (and thus $T$ is not Ramsey-cofinite).
%\end{enumerate}
%\end{theorem}

We already know from \cref{ex:backward-half-graph,ex:forward-half-graph} that if $T$ is weakly expanding or has a finite dominating set, then $T$ is not Ramsey-lower-dense (and thus $T$ is not Ramsey-confinite).  So all that remains to prove \cref{thm:cofinitetree}(ii) is to show that every forest in $\CT^{*}$ is not Ramsey-lower-dense.  Recall that $\CT^*$ is the family of forests $T$ having one vertex $t$ of infinite degree, every other vertex has degree at most $d$ for some $d\in \NN$, $t$ is adjacent to infinitely many leaves and infinitely many non-leaves, and cofinitely many vertices of $T$ have distance at most $2$ to $t$ (in particular, if $T$ is not connected, then $T$ has one infinite component and finitely many finite components).

\begin{proposition}\label{lem:T-star}
If $T \in \CT^{*}$, then $T$ is not Ramsey-lower-dense.
\end{proposition}

\begin{proof}
Let $T\in \CT^{*}$, let $t$ be the vertex of infinite degree in $T$, and let $d$ be the maximum degree of $T-\{t\}$ as guaranteed by the definition.

We begin by partitioning $\NN$ into intervals $A_0, A_1, A_2, \dots$ as follows: Let $a_0=1$ and for all $i\geq 1$, let $a_i=i\cdot d\cdot a_{i-1}$.  Then for all $i\geq 0$, let $A_i=[a_i, a_{i+1})$.  Now for all $r\in \{0,1,2,3\}$, let $V_r=A_r\cup A_{4+r}\cup A_{8+r}\cup \dots$.

Now color all edges which are inside $V_0$ or $V_1$ red, and all edges inside $V_2$ or $V_3$ blue. Color all edges between $V_0$ and $V_1$ blue and all edges between $V_2$ and $V_3$ red. Finally, for all $i\in \{0,1\}$, $j\in \{2,3\}$ we color the complete bipartite graphs $K(V_i,V_j)$ according to \cref{fig:T-star} as follows: Suppose first that there is a red arrow from $V_i$ to $V_j$ (and thus a blue arrow from $V_j$ to $V_i$). Let $A_s\subseteq V_i$ and $A_t\subseteq V_j$.  If $s<t$, then color all edges between $A_s$ and  $A_t$ red.  If $t<s$, then color all edges between $A_t$ and $A_s$ blue.  If there is a blue arrow from $V_i$ to $V_j$ (and thus a red arrow from $V_j$ to $V_i$), we do the opposite.

\begin{figure}[ht]
	\begin{center}
		\includegraphics[height=5cm]{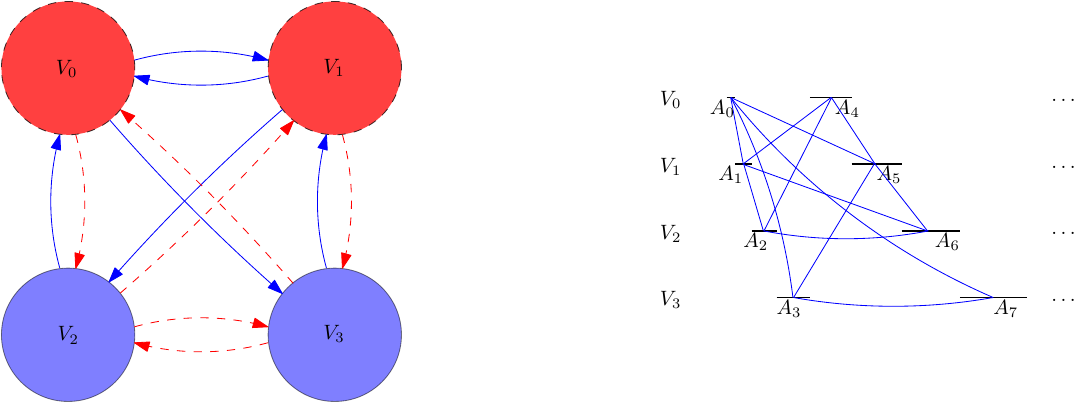}
	\end{center}
	\caption{The shaded areas denote cliques of the respective colors and a blue/solid (red/dashed) arrow from $V_i$ to $V_j$ indicates that vertices in $V_i$ have cofinitely many blue (red) neighbors in $V)j$. On the right we have an example of the relevant edges in the case where we are embedding a blue copy of $T\in \CT^*$ with root $t$ in $V_0$.}
	\label{fig:T-star}
\end{figure}

Assume there is a monochromatic copy $T'$ of $T$, let $f$ be the corresponding embedding, and let $V_i'=f(T)\cap V_i$ for all $i\in \{0,1,2,3\}$. By symmetry, we may assume that $t$ is embedded in $V_0$; and let $v = f(t)$.

Suppose first that $T$ is embedded in the blue subgraph.  Since $v$ has finitely many blue neighbors in $V_2$ and cofinitely many vertices have distance at most 2 to $v$, all but finitely many vertices of $V_2'$ are neighbors of vertices in $V_1'$ in $f(T)$. So for all sufficiently large $q$ (i.e. large enough so that $N_{T'}(v)\cap V_2\subseteq A_2\cup A_6\cup \dots \cup A_{4q-2}$), we have $$|V_2'\cap A_{4q+2}|\leq d|V_1'\cap (A_{1}\cup A_5\cup \dots \cup A_{4q+1})|\leq d\cdot a_{4q+1}\leq \frac{a_{4q+2}}{4q+2},$$
and thus
\begin{align*}
\ld(\ran f)
  &\leq \lim_{q \to \infty} |\ran f \cap (A_1 \cup \ldots \cup A_{4q+2})|/(a_1 + \ldots + a_{4q+2})\\
  &\leq \lim_{q \to \infty} (a_1 + \ldots + a_{4q+1} + \tfrac{a_{4q+2}}{4q+2})/(a_1 + \ldots + a_{4q+2})=0.
\end{align*}

Suppose next that $T$ is embedded in the red subgraph.  Since $v$ has no red neighbors in $V_1$, all but finitely many vertices of $V_1'$ are neighbors of vertices in $V_2'$ in $f(T)$. So for all $q\geq 1$, we have $$|V_1'\cap A_{4q+1}|\leq d|V_2'\cap (A_{2}\cup A_6\cup \dots \cup A_{4q-2})|\leq d\cdot a_{4q-2}\leq \frac{a_{4q+1}}{4q+1},$$
and thus $\ld(f(T))=0$.  Since $f$ was an arbitrary embedding this shows that $T$ is not Ramsey lower dense.
\end{proof}

We now turn to part (i) of \cref{thm:cofinitetree}; that is, if $T$ is a forest which is strongly contracting, has no finite dominating set, and $T\not\in\CT^*$, then $T$ is Ramsey-cofinite.  We begin with a lemma which allows us to embed forests which are strongly contracting into graphs with infinitely many vertices of cofinite degree. Here, it is not important that we are embedding forests and we will state and prove the lemma more generally.  Recall that a graph $F$ is strongly contracting if there exists $k\in \NN$ such that for all $\ell\in \NN$ there exists an independent set $A$ in $F$ with $|A|\geq \ell$ such that $|N(A)|\leq k$ (and in particular, a forest is strongly contracting if and only if it has finitely many components and unbounded leaf degree).  

\begin{lemma}\label{lem:embed-leafy-tree}
	Let $F$ be a graph. A cofinite copy of $F$ can be found in every graph $H$ having infinitely many vertices of cofinite degree if and only if $F$ is strongly contracting.
\end{lemma}

In order to simplify the proof of the lemma, we first prove a structural result regarding strongly contracting graphs.

\begin{proposition}\label{prop:sc-structure}
If $F$ is strongly contracting, then there exists non-negative integers $\ell\leq k$, an infinite independent set $U$, a family of disjoint sets $\{U_i\subseteq U: i\in \NN\}$ with $|U_i|\geq i$ for all $i\in \NN$, an $\ell$-set $V_0$, and a family of disjoint $(k-\ell)$-sets $\{V_i\subseteq V(F)\setminus (U\cup V_0): i\in \NN\}$ such that for all $i\in \NN$, $F[U_i, V_0\cup V_i]$ is a complete bipartite graph (note that the family $\{V_0\cup V_i: i\in \NN\}$ is a $k$-uniform \emph{sunflower} with a core of order $\ell$).
\end{proposition}

\begin{proof}
Since $F$ is strongly contracting, there exists a non-negative integer $k'$ and without loss of generality there exists disjoint independent sets $U_1', U_2', \dots$ such that for all $i\in \NN$, $|U_i'|\geq i2^{k'}$ and $|N(U_i')|\leq k'$.  Now by pigeonhole, there exists a subset $U_i\subseteq U_i'$ with $|U_i|\geq i$ and a non-empty subset $V_i\subseteq N(U_i)$ such that $F[U_i, V_i]$ is a complete bipartite graph.  Indeed, since $|N(U_i')|\leq k$, there are at most $2^{k'}$ different possible neighborhoods for vertices in $U_i'$.  By pigeonhole, there are at least $i$ vertices in $U_i'$ with the same neighborhood.  Now by pigeonhole again, there exists an integer $0\leq k\leq k'$ and an infinite subsequence $\{i_j\}_{j\in \NN}$ such that $|U_{i_j}|\geq i_j$ and $|V_{i_j}|=k$ and $F[U_{i_j}, V_{i_j}]$ is a complete bipartite graph.  Finally, by the infinite version of the sunflower lemma (see \cite[Page 107, Problem 1]{KT}\footnote{The proof is short.  If $k=1$, then it is clear, so let $k\geq 2$ and suppose it is true for $(k-1)$-uniform hypergraphs.  If there is an infinite matching we are done; so suppose not.  Since there is no infinite matching, there is a finite vertex cover and thus a vertex $v$ of infinite degree.  Now apply induction to the $(k-1)$-uniform link graph of $v$.}), there exists a subsequence $\{j_h\}_{h\in \NN}$ such that the family $\{V_{i_{j_h}}:h\in \NN\}$ is a sunflower.
\end{proof}

\begin{proof}[Proof of \cref{lem:embed-leafy-tree}]
Note that in Example \ref{ex:forward-half-graph}, both the red graph and the blue graph have the property that there are infinitely many vertices of cofinite degree.  So if a cofinite copy of $F$ can be found in every graph $H$ having infinitely many vertices of cofinite degree, then $F$ is not weakly expanding; i.e.\ $F$ is strongly contracting.

Now suppose $F$ is strongly contracting.  By \cref{prop:sc-structure}, there exists non-negative integers $k$ and $\ell$, an infinite independent set $U$, a family of disjoint sets $\{U_i\subseteq U: i\in \NN\}$ with $|U_i|\geq i$ for all $i\in \NN$, an $\ell$-set $V_0$, and a disjoint family of $k$-sets $\{V_i\subseteq V(F)\setminus (U\cup V_0): i\in \NN\}$ such that for all $i\in \NN$, $F[U_i, V_0\cup V_i]$ is a complete bipartite graph.

Since there are infinitely many vertices $V'$ in $H$ with cofinite degree, we may choose an infinite clique $K\subseteq H$ such that $V(K)\subseteq V'$. If $V(H)\setminus V(K)$ is finite, then we are done (since it is clear that we can surjectively embed $F$ into the clique $K$), so suppose not.
Let $X\subseteq V(K)$ such that $X$ and $V(K)\setminus X$ are both infinite.  Let $y_1, y_2, \dots$ be an enumeration of $Y:=\NN \setminus V(K)$.  For all $i\in \NN$, let $X_i=\{x\in X: \forall j\geq i, \{x, y_j\}\in E(H)\}$.  Note that if there exists $i\in \NN$ such that $X_i$ is infinite, then we have a copy of $K_{\N, \N}$ which covers all but finitely many vertices in $Y$ and from here we can easily get the desired embedding.  So suppose that $X_i$ is finite for all $i\in \NN$.  For all $x\in X$, there exists $i\in \N$ such that $x\in X_i$, so let $\phi(x)=\min\{i\in \N: x\in X_i\}$.  Let $x_1, x_2, \dots$ be an enumeration of $X$ such that for all $i\leq j$, $\phi(x_i)\leq \phi(x_j)$.  Let $X_0'=\{x_1, \dots, x_\ell\}$ and for all $i\geq 1$, let $X_i'=X_0'\cup \{x_{\ell+(i-1)k+1}, \dots, x_{\ell+ik}\}$ and let $\phi(i)=\phi(x_{\ell+ik})$.  Finally, for all $i\geq 1$, let $Y_i=\{y_{\phi(i)}, \dots, y_{\phi(i+1)-1}\}$ (note that $Y_i=\emptyset$ if and only if $\phi(i)=\phi(i+1)$).  Let $I=\{i\geq 1: Y_i\neq \emptyset\}$ and note that since every $X_i$ is finite, we have that $I$ is infinite.  Also note that for all $i\in I$, $[X_i', Y_i]$ is a complete bipartite graph.

From the properties of $F$ mentioned above, there is an injection $f:I\to \NN$ such that $\NN\setminus f(I)$ is infinite and for all $i\in I$, $|U_{f(i)}|\geq |Y_i|$ and thus there exists $U'_{f(i)}\subseteq U_{f(i)}$ such that $|U'_{f(i)}|=|Y_i|$.  Now for all $i\in I$, we embed $U'_{f(i)}$ to $Y_i$ and embed $V_{f(i)}$ to $X'_{i}$.  Since $\NN\setminus f(I)$ is infinite, there are infinitely many remaining vertices in $F$ and these can be embedded to $K-X$ (since $K$ is a clique and $U$ is an independent set).
\end{proof}

Note that \cref{lem:embed-leafy-tree} allows us to focus on colorings in which all but finitely many vertices have infinite degree in both colors.  Given such a coloring, we will need to separate into a few cases depending on the structure of $T$.

\begin{proposition}\label{fact:types}
	Let $T$ be an infinite tree with unbounded leaf degree, no finite dominating set, and $T\not\in \CT^{*}$.  Then either $T$ has unbounded radius, or there exists a vertex $t \in V(T)$ such that $t$ is adjacent to infinitely many non-leaves and
	\begin{enumerate}
		\item there are infinitely many paths of length $3$ starting at $t$ which are pairwise vertex-disjoint apart from $t$, or
		\item there is a vertex in $T - \{t\}$ of infinite degree, or
		\item the neighbors of $t$ have unbounded degrees.
	\end{enumerate}
\end{proposition}

\begin{proof}
	Suppose that $T$ has unbounded leaf degree, no finite dominating set, bounded radius, and $T \not\in \CT^{*}$.  We first show that there is a vertex $t \in V(T)$ which is adjacent to infinitely many non-leaves.
	%Since $T$ has bounded radius, there must be vertex $s$ of infinite degree.
	Let $s_0 \in V(T)$ and think of it as the root. If $s_0$ is adjacent to infinitely many non-leaves, we are done; so assume it is adjacent to finitely many non-leaves $S_1$. Since $T$ has no finite dominating set, $V(T) \setminus N(s_0)$ is infinite. In particular, some $s_1 \in S_1$ has an infinite subtree. If $s_1$ is adjacent to infinitely many non-leaves, we are done; so assume it is adjacent to finitely many non-leaves $S_2$ other than $s_1$. Since there is no finite dominating set, some $s_2 \in S_2$ has an infinite subtree. We keep iterating until we find a vertex $s_i$ adjacent to infinitely many non-leaves. Since $T$ has bounded radius, this process must finish eventually.

	Let $t \in V(T)$ which is adjacent to infinitely many non-leaves and let $S$ be the non-leaves adjacent to $t$.  Assume that there are only finitely many paths of length $3$ starting at $t$ which are pairwise vertex-disjoint apart from $t$ and that there is no $t' \in V(T) \setminus \{t\}$ of infinite degree (otherwise we are done).
	Let $S' \subset S$ be those vertices whose only children are leaves and let $S'' = S \setminus S'$.  By the assumption, we have that $S''$ is finite.  
	We claim that for each $i \in \N$, there is a vertex $t' \in S'$ of degree at least $i$. Assume for contradiction this is not the case and let $d := \max_{s \in S'} \deg(s)$.	
	Note that, since $T$ has bounded radius and no vertex of infinite degree other than $t$, there are only finitely many vertices which are successors of vertices in $S''$.  
	Thus, since $T$ has unbounded leaf degree, $t$ must be adjacent to infinitely many leaves. Therefore $T \in \CT^{*}$, a contradiction.
\end{proof}

The main difficulty in the proof of \cref{thm:cofinitetree} will be dealing with the trees of bounded radius described in \cref{fact:types}.  The following lemma deals with that case.

\begin{lemma}\label{lem:embed-all-other}
	Let $H$ be a $2$-colored $K_\NN$ in which every vertex has infinite degree in both colors.
	Let $T$ be a tree of bounded radius containing a vertex $t$ such that $t$ is adjacent to infinitely many non-leaves and
	\begin{enumerate}
		\item there are infinitely many paths of length $3$ starting at $t$ which are pairwise vertex-disjoint apart from $t$, or
		\item there is a vertex in $T - \{t\}$ of infinite degree, or
		\item the neighbors of $t$ have unbounded degrees.
	\end{enumerate}
	Then there is a cofinite monochromatic copy of $T$ in $H$.
\end{lemma}

Assuming \cref{lem:embed-all-other} (the proof of which we delay for the moment) we now prove \cref{thm:cofinitetree}.

\begin{proof}[Proof of \cref{thm:cofinitetree}]
	Part (ii) follows from \cref{ex:backward-half-graph,ex:forward-half-graph,lem:T-star}.

	So let $T$ be a forest which is strongly contracting, has no finite dominating set, and $T\not \in \CT^{*}$. 

	Let $H$ be a $2$-colored $K_\N$, and assume the colors are red and blue. If there are infinitely many vertices of cofinite red degree or infinitely many vertices of cofinite blue degree, then we are done by Lemma \ref{lem:embed-leafy-tree}; so (by removing finitely many vertices) we may assume every vertex in $H$ has infinite red degree and infinite blue degree.

	First, suppose that $T$ has a component of unbounded radius or infinitely many components. If the blue subgraph $H_B$ is infinitely connected, we can find a monochromatic spanning copy of $T$ in $H_B$ by \cref{lem:embed-deep-tree}. If $H_B$ is not infinitely connected, there exists a finite set $X$ such that $H_B-X$ is not connected. So there exists a partition $\{Y, Z\}$ of $\NN-X$ such that all edges between $Y$ and $Z$ are red. Note that for all $v\in Y$, $N_B(v)\subseteq Y\cup X$ and for all $v\in Z$, $N_B(v)\subseteq Z\cup X$.  Since all vertices have infinite blue degree and $X$ is finite, this implies that both $Y$ and $Z$ are infinite.  So $Y\cup Z$ is cofinite and $H_R[Y,Z]$ induces a red copy of $K_{\NN, \NN}$. Since $T$ has no finite dominating set, both parts of its bipartition are infinite, and we can surjectively embed $T$ into $H_R[Y,Z]$.

	Finally, suppose that $T$ has finitely many components $T_1, \dots, T_k$ all of which have bounded radius.  For all $i\in [k]$, let $t_i\in V(T_i)$ and for all $2\leq i\leq k$, add the edge $t_1t_k$ to get a tree $T'$ which has bounded radius, is strongly contracting, has no finite dominating set, $T'\not\in \CT^{*}$, and $T'$ contains $T$ as a spanning subgraph.  Therefore, by \cref{fact:types}, $T'$ satisfies the hypotheses of \cref{lem:embed-all-other} and thus we can find a monochromatic copy of $T'$ (and consequently $T$) which spans cofinitely many vertices of $H$.
\end{proof}

It remains to prove \cref{lem:embed-all-other}.

\begin{proof}[Proof of \cref{lem:embed-all-other}]
	Let $T$ and $t$ be as in the statement and let $H$ be a $2$-colored $K_{\NN}$ which is as in the statement. Note that $\deg(t) = \infty$ and there are infinitely many paths of length $2$ starting at $t$ which are vertex disjoint apart from $t$ ($\star$).
	Let $v \in \N$ be an arbitrary vertex.
	Let $B$ be the set of blue neighbors of $v$ and let $R$ be the set of red neighbors of $v$. Furthermore, let $B'$ be the set of vertices in $B$ with finitely many red neighbors in $R$ and let $R'$ be the set of vertices in $R$ with finitely many blue neighbors in $B$. Let $R'' = R \setminus R'$ and $B'' = B \setminus B'$.

	\noindent\textbf{Case 1} ($B'$ or $R'$ is finite.)

		Suppose without loss of generality that $R'$ is finite. We will find an embedding $f$ of $T$ into the blue subgraph $H_B$ covering the cofinite set $\{v\}\cup B \cup R''$. We build $f$ iteratively in finite pieces maintaining a partial embedding of $T$ whose domain is connected. Note that, by keeping $\dom f$ connected, we ensure that every not yet embedded vertex is adjacent to at most one vertex in $\dom f$.
		Initially, we set $f(t) = v$. Then we repeatedly follow the following two steps.
		\par\smallskip
		\noindent\underline{Step 1.} Let $s \in V(T) \setminus \dom f$ be the smallest not yet embedded vertex. Since every vertex has infinite blue degree, it is easy to extend $f$ so that $\dom f$ remains connected and $s \in \dom f$ (by adding a path to $s$).
		\par\smallskip
		\noindent\underline{Step 2.} Let $u \in (B \cup R'') \setminus \ran f$ be the smallest not yet covered vertex. If $u \in B$ we can simply choose a not-yet embedded neighbor of $t$ and embed $u$ into it (since $t$ has infinite degree). If $u \in R''$, it has infinitely many blue neighbors in $B$ (by definition of $R'$) and therefore there is a blue path $vxu$ of length $2$ for some $x \in B\setminus \ran f$. By ($\star$) we can extend $f$ to cover $x$ and $u$ so that $\dom f$ remains connected.
		\par\smallskip
		Routinely, this defines an embedding of $T$ into $H_B$ covering $B \cup R'$.

	\noindent\textbf{Case 2} ($B'$ and $R'$ are infinite.)

		We further split into subcases depending on the structure of $T$.
		\par\smallskip
		\textbf{Case 2.1} ($T$ has infinitely many paths of length $3$ starting at $t$ which are pairwise vertex-disjoint apart from $t$):
		We will find an embedding $f$ of $T$ into the blue subgraph $H_B$ covering $\NN$.
		We build $f$ iteratively in finite pieces maintaining a partial embedding of $T$ whose domain is connected.
		Initially, we set $f(t) = v$. Then we repeat the following two steps.
		\par\smallskip
		\noindent\underline{Step 1.} We extend $f$ to include the smallest not yet embedded vertex as above.
		\par\smallskip
		\noindent\underline{Step 2.} Let $u \in \NN \setminus \ran f$ be the smallest not yet covered vertex. If $u \in B \cup R''$ we proceed as above; so suppose $u \in R'$.
		Note that $u$ (as every other vertex) has infinitely many blue neighbors, and by definition of $R'$ only finitely many of those can lie outside $R$. Furthermore, every vertex $ u' \in B'$ has only finitely many red neighbors in $R$ and thus $u$ and $u'$ have infinitely many common blue neighbors in $R$. It follows that there is a blue path $P$ of length $3$ from $v$ to $u$ so that $(P \setminus \{v\}) \cap \ran f = \emptyset$. By the case assumption we can extend $f$ to cover this path.
		\par\smallskip
		Routinely, the resulting function $f$ is an embedding of $T$ into $\NN$. Observe that the only difference to the previous case is how we deal with vertices in $R'$. This will be similar in the following cases and we therefore skip some details.
		\par\smallskip
		\textbf{Case 2.2} ($T$ has a vertex $t' \in V(T) \setminus \{t\}$ of infinite degree):
		Given some integer $d \geq 1$, we say that a path $P$ is $d$-good in blue (red) if it has length $d$, starts at $v$ and ends at some $v' \in \N \setminus \{v\}$, is monochromatic in blue (red) and $v'$ has only finitely many red (blue) neighbors in $R'$ ($B'$).
		We will first show that (i) for every positive integer $d \not = 2$, there is a red $d$-good path and a blue $d$-good path, and (ii) there is a red 2-good path or a blue 2-good path.

		If $d = 1$, any vertex in $B'$ forms a blue $d$-good path with $v$. Since any two vertices in $B'$ have infinitely many common blue neighbors in $R$, we can extend this path to a $3$-good path. We can proceed like this to get a $d$-good path in blue path for any odd $d$. If $d \geq 4$ is even, we start by building a blue path of length $2$ to some $u \in R'$ and take some $v' \in B'$ not yet in the path. Since $u$ has infinitely many blue neighbors in $R$ and every vertex in $B'$ has only finitely many red neighbors in $R$, $u$ and $v'$ have a common blue neighbor not yet in the path, giving us a $d$-good path in blue. We can extend this path now as before to any even length $d \geq 4$. We can proceed similarly for red paths.
		Finally suppose $d=2$.  If there are $u_1 \in R$ and $u_2 \in R'$ such that $u_1u_2$ is red, then $vu_1u_2$ is $2$-good in red and we are done. Otherwise every $u \in R$ has only blue neighbors in $R'$. We can thus form a blue path from $v$ to $R$ of length $2$, which is $2$-good in blue.

		Let now $d$ be the distance from $t$ to $t'$ in $T$ and let $P$ be a $d$-good path (say in blue). Embed $t$ into $v$ and the unique path from $t$ to $t'$ into $P$ (and call this partial embedding $f$). Let $v' = f(t')$ and remove the finite set of vertices in $R'$ which is not in the blue neighborhood of $v'$. We then extend $f$ in finite pieces exactly as in the previous case apart from when $u \in R'$, where we simply embed an available neighbor of $t'$ into $u$.
		\par\smallskip
		\textbf{Case 2.3} (for all $i \in \N$, there is a vertex $t' \in N_T(t)$ of degree at least $i$):
		We may assume we are not in Case 2.2 and thus there is an infinite set $S \subset N_T(t)$ of vertices with distinct degrees.  Furthermore, we may assume we are not in Case 2.1 and thus cofinitely many vertices of $T$ have distance at most 2 to $t$.  Let $T_0$ be the finite subtree rooted at $t$ which consists of all paths from $t$ to leaves of distance at least $3$.

		Let $u_1, u_2, \dots$ be an enumeration of $N_T(t)$.  Let $y_1, y_2, \dots$ be an enumeration of $R$.  Let $B_1\subseteq B'$ such that $B_1$ is infinite and $B'\setminus B_1$ is infinite, then set $B_2=B\setminus B_1$.

		For all $x\in B_1$ there exists $\phi(x) \in \NN$ such that $y_{\phi(x)-1}\not\in N_B(x)$ and $y_j\in N_B(x)$ for all $j\geq \phi(x)$ (if $ Y \subset N_B(x)$, then $\phi(x) = 1$).
For all $i\geq 1$, let $X^i=\{x\in B_1:\phi(x)=i\}$.  If $|X^i|=\infty$ for some $i\in \NN$, then reset $B_1:=X^i$ and $B_2:=B\setminus B_1$, enumerate $B_1$ as $x_1, x_2, \dots$.  Otherwise $X^i$ is finite for all $i$ and thus there is a natural enumeration $x_1, x_2, \dots$ of $B_1$ such that $\phi(x_i) \leq \phi(x_j)$ whenever $i \leq j$.

		Initially we set $f(t)=v$ and we embed $T_0$ in $H_B$ using the fact that every vertex in $H$ has infinite blue degree.  Now every vertex in $V(T)\setminus \dom f$ has distance at most 2 to $r$.  Now we repeat the following two steps.

		\par\smallskip
		\noindent\underline{Step 1.}
		Let $x_i$ and $x_j$ (with $i<j$) be the two smallest vertices in $B_1\setminus \ran f$ and let $u_{n_i}\in N_T(t)\setminus \dom f$ such that $u_{n_i}$ has at least $\phi(x_{j})-\phi(x_i)$ children in $T$ (which is possible by the case). Set $f(u_{n_i})=x_i$ and embed the (finitely many) vertices in $N_T(u_{n_i})\setminus \{t\}$ to the smallest vertices in $\{y_{\phi(x_i)}, y_{\phi(x_i)+1}, \dots\}\setminus \ran f$.

		\par\smallskip
		\noindent\underline{Step 2.}
		Injectively embed all vertices from $\{u_1, u_2, \dots, u_{n_i+1}\}\setminus \dom f$ (which is non-empty since $u_{n_i+1}\not\in \dom f$) to the smallest vertices in $B_2\setminus \ran f$.  Now, using the fact that every vertex in $H$ has infinite blue degree, iteratively embed the children of each vertex $u_{n_\ell}\in \{u_{n_{i-1}+1}, \dots, u_{n_i-1}\}$ anywhere in $N_B(f(u_{n_\ell}))\setminus (\{x_{j}\}\cup \ran f)$.  Now move to Step 1 (and notice that $x_j$ will become the smallest vertex in $B_1\setminus \ran f$).

		\par\smallskip
		The resulting function $f$ is an embedding of $T$ into $H$ covering a cofinite set.
\end{proof}

\subsection{General graphs}\label{sec:cofingen}
In the previous section, we completely characterize forests which are Ramsey-cofinite.  We know that if a graph $F$ is Ramsey-cofinite, then $F$ is bipartite, strongly contracting, and has no finite dominating set (by \cref{ex:forward-half-graph,ex:backward-half-graph,ex:KNN}).  On the other hand, from the proof in the previous section we know that if $G$ is bipartite, strongly contracting, and has no finite dominating set and we are given a 2-coloring of $K_{\NN}$ such that one of the colors is not infinitely connected, then there is a cofinite monochromatic copy of $G$.  So this raises the question of completely characterizing all graphs which are Ramsey-cofinite.  However, given the information from the previous sections, we can narrow this down to a much more specific question.

\begin{problem}\label{q:specific}
Characterize the graphs $G$ which are bipartite, strongly contracting, and have no finite dominating set such that there exists a cofinite monochromatic copy in every 2-coloring of $K_{\NN}$ in which both colors are infinitely connected.
\end{problem}

The following is an easy to state sufficient condition (we are aware of a more general sufficient condition which contains \cref{thm:cofinitetree}(i), but as we don't believe the more general condition is necessary, we go for simplicity instead).

\begin{theorem}
If $G$ is bipartite, strongly contracting, and has arbitrarily long paths whose internal vertices have degree $2$, then $G$ is Ramsey-cofinite.
\end{theorem}

This follows because if we are given a 2-coloring of $K_{\NN}$ in which both colors are infinitely connected, then at least one of those colors contains an infinite clique.  So we can use the following lemma.

\begin{lemma}\label{lem:embed-induced-paths}
Let $F$ be a connected graph.  A spanning copy of $F$ can be found in every infinitely connected graph $H$ with an infinite clique if $F$ has arbitrarily long paths whose internal vertices have degree $2$.
\end{lemma}

\begin{proof}
Assume that $F$ has arbitrarily long induced paths and that $H$ is infinitely connected with an infinite clique $K \subseteq H$. We will construct an embedding $f$ of $F$ into $H$ iteratively in finite pieces. For each $i \in \N$, we will do the following two steps: First, let $t \in V(T) \setminus \dom f$ be the smallest not-yet embedded vertex and embed it into an arbitrary vertex $u \in V(K) \setminus \ran f$.
Second, let $v \in V(H) \setminus \ran f$ be the smallest not-yet covered vertex and let $P$ be a finite path in $H$ which starts and ends in $K$, contains $v$ and avoids $\ran f$ (such a path exist since $H$ is infinitely connected). Let $P'$ be an induced path in $T$ of the same length as $P$ which avoids $\dom f$ (such a path exists since $T$ has arbitrary long induced paths). Extend $f$ by embedding $P'$ into $P$. Note that all neighbors of the internal vertices of $P'$ will be embedded, and the endpoints of $P'$ are in $K$.
Therefore, we maintain a partial embedding throughout the process. Since we eventually embed every $t \in V(T)$, the resulting function $f$ is an embedding of $T$ into $H$. Since we eventually cover every $v \in V(H)$, this embedding is surjective.
\end{proof}

\section{Ramseyness of coideals}\label{sec:coideals}

\subsection{Ideals and coideals}

An \emph{ideal} on a set $X$ is a collection $\mathcal{I}$ of subsets of $X$
such that (1) for any $B\in \mathcal{I}$ and $A\subseteq B$, we have
$A\in\mathcal{I}$, and (2) for any $A,B\in\mathcal{I}$, we have $A\cup
B\in\mathcal{I}$.  We call an ideal $\mathcal{I}$ on $X$ \emph{proper} if
$X\not\in\mathcal{I}$.  If $\mathcal{I}$ is an ideal, then we write
$\mathcal{I}^+$ for its complement $\mathcal{P}(X) \setminus \mathcal{I}$, and
we call $\mathcal{I}^+$ a \emph{coideal}.

In this section we will primarily be concerned with ideals on countable sets,
and in particular ideals on $\NN$.  Some commonly used examples of ideals on
$\NN$ are
\begin{enumerate}
  \item $\fin = \set{A\subseteq\NN}{|A| < \infty}$,
  \item $\mathcal{Z}_0 = \set{A\subseteq\NN}{d(A) = 0}$,
  \item $\mathcal{I}_{1/n} = \set{A\subseteq\NN}{\sum_{n\in A} 1/n < \infty}$.
\end{enumerate}

In general, we view an ideal $\mathcal{I}$ on $X$ as a way of measuring which
subsets of $X$ are ``small''.  In this light, we only consider an ideal
$\mathcal{I}$ to be nontrivial if $\mathcal{I}$ is proper and contains the
finite subsets of $X$, since at the very least, the finite subsets of $X$ should
be ``small'', and $X$ itself should not be ``small''.

Let $G$ be a graph and $\mathcal{I}^+$ be a coideal on $\NN$.  We say that $G$ is \emph{$\mathcal{I}^+$-Ramsey} if, for every finite coloring of $K_\NN$, there is a monochromatic copy of $G$ whose vertex set is in $\mathcal{I}^+$.

The first thing to note is that we may reexpress one of the
central notions of this paper using the above terminology; namely, a graph $G$
is Ramsey-dense if and only if $G$ is $\mathcal{Z}_0^+$-Ramsey.  For another example, Ramsey's theorem says that $K_{\N}$ is $\fin^+$-Ramsey.  In the remainder of this section we investigate the relationship between coideals
$\mathcal{I}^+$ and graphs $G$ such that $G$ is $\mathcal{I}^+$-Ramsey.  We hope
that the results to follow will help the reader to better understand some of the
characteristics of Ramsey-dense graphs, while simultaneously establishing a more
general setting for the kind of questions we are interested in, where different
notions of ``small'' other than ``asymptotic density zero'' are considered.
Before continuing we note three easy observations.

\begin{fact}\label{fact:subset}
Let $\mathcal{I}$ and $\mathcal{J}$ be ideals on $\NN$ and suppose $\mathcal{I}\subseteq\mathcal{J}$.  For any graph $G$, if $G$ is $\mathcal{J}^+$-Ramsey, then $G$ is $\mathcal{I}^+$-Ramsey.
\end{fact}

\begin{fact}\label{fact:part}
Let $\mathcal{I}$ be an ideal on $\NN$ and let $A\in \mathcal{I}^+$.  For any partition $\{A_1, \dots, A_k\}$ of $A$, there exists $i\in [k]$ such that $A_i\in \mathcal{I}^+$.
\end{fact}

\begin{fact}\label{fact:coideal-contains-ultrafilter}
For every nontrivial ideal $\CI$, there is an ultrafilter $\SU \subset \CI^+$.
\end{fact}

\begin{proof}
Let $\CF = \{I^c: I \in \CI\}$ and observe that $\CF$ satisfies $(i)-(iii)$ in \cref{def:ultrafilter} (such a family is called a \emph{filter}). Hence we can apply Zorn's lemma as in \cref{prop:non-trivial-uf} to get an ultrafilter $\SU$ containing $\CF$. Then, for every $I \in \CI$, we have $I^c \in \CF \subset \SU$ and thus $I \not\in \SU$; hence $\SU \subset \CI^+$.
\end{proof}

\subsection{Finitely-ruled graphs and the ideal \texorpdfstring{$\nwd$}{nwd}}
Recall that one of the motivating problems of this paper is to characterize the
Ramsey-dense graphs, or in other words the graphs $G$ such that $G$ is
$\mathcal{Z}_0^+$-Ramsey.  In Theorem~\ref{thm:nwd} we provide a
characterization of those graphs $G$ for which $G$ is $\CI^+$-Ramsey for \emph{every} nontrivial ideal $\CI$
on $\NN$.  Interestingly, this characterization reduces to one
particular ideal, and one particular $2$-coloring, both of which we will
describe now.

Note that every positive integer $n$ has a binary expansion in which the leftmost digit is a 1, the \emph{truncated binary expansion} of $n$ is what remains after removing the leftmost digit from the binary expansion (for instance, the truncated binary expansion of 19 is 0011).
Given $s,t\in \N$, we say that $t$ \emph{extends} $s$, if $s\leq t$ and the truncated binary expansion of $t$ contains the truncated binary expansion of $s$ as its initial segment (for instance $19$ extends $7$ since $0011$ contains $11$ as its initial segment, reading from right to left).  Given $s\in \N$, we write $\langle s\rangle$ for the set of $t\in \N$ which extend $s$ (for instance, $\langle 1 \rangle=\N$, $\langle 2\rangle$ is the positive even integers, and $\langle 3\rangle$ is the positive odd integers).  The ideal $\nwd$ consists of all sets $A\subseteq \N$ such that for every
$s\in \N$ there exists an extension $t$ of $s$ such that $A\cap \lrang{t} =
\emptyset$.\footnote{The notation $\nwd$ stands for ``nowhere dense'' and typically the ideal $\nwd$ is studied on the set $\QQ$; however, for consistency with the rest of the paper, we state all of the results in terms of the set $\NN$.  That being said, it is possible to show that the set
$\N$, when given the topology generated by the sets $\lrang{s}$, is homeomorphic
to the space $\QQ\cap [0,1)$, and under this homeomorphism the sets in $\nwd$
correspond to those subsets of $\QQ\cap [0,1)$ which are not dense in any
subinterval of $[0,1)$.}  It is straightforward to check that $\nwd$ is a non-trivial ideal (that is, a proper ideal containing all of the finite subsets of $\N$).

%We write $2^{<\NN}$ for the set of all finite binary sequences, which we view as
%functions from $\{0,\ldots,n-1\}$, for some $n\in\NN$, to $\{0,1\}$.  If $s\in
%2^{<\NN}$, we write $|s|$ for its length.  Given $s,t\in 2^{<\NN}$, we say that
%$t$ \emph{extends} $s$ if $|s| \le |t|$ and $t(i) = s(i)$ for all $i < |s|$.
%Finally, for each $s\in 2^{<\NN}$ we write $[s]$ for the set of all $t\in
%2^{<\NN}$ which extend $s$.

Recall that the Rado coloring was defined in \cref{sec:rado}.

%The ideal $\nwd$ consists of all sets $A\subseteq 2^{<\NN}$ such that for every
%$s\in 2^{<\NN}$ there exists an extension $t$ of $s$ such that $A\cap [t] =
%\emptyset$.  \footnote{The notation $\nwd$ stands for ``nowhere dense''; the set
%$2^{<\NN}$, when given the topology generated by the sets $[s]$, is homeomorphic
%to the space $\QQ\cap [0,1]$, and under this homeomorphism the sets in $\nwd$
%correspond to those subsets of $\QQ\cap [0,1]$ which are not dense in any
%subinterval of $[0,1]$.}  We note that $\nwd$ is a proper ideal containing all
%of the finite subsets of $2^{<\NN}$.

\begin{theorem}\label{thm:nwd}
  The following are equivalent for any countably infinite graph $G$.
  \begin{enumerate}[label={(\roman*)}]
    \item \label{nwd1} For every non-trivial ideal $\mathcal{I}$ on $\NN$, $G$ is $\mathcal{I}^+$-Ramsey.
    \item \label{nwd2} In the Rado coloring of $K_{\N}$, there is a monochromatic copy of $G$ such that $V(G)\in \nwd^+$.
    \item \label{nwd3} $G$ is finitely-ruled.
  \end{enumerate}
\end{theorem}

\begin{proof}[Proof of \cref{thm:nwd}]
    (\ref{nwd1}$\implies$\ref{nwd2}) $\nwd$ is a non-trivial ideal on $\NN$, so in every $2$-coloring of the edges of $K_{\NN}$ (and in particular, the Rado coloring), there is a monochromatic copy of $G$ with $V(G)\in \nwd^+$.

    (\ref{nwd2}$\implies$\ref{nwd3})
	Suppose $G$ is infinitely-ruled and there exists a monochromatic copy of $G$ in the Rado coloring of $K_{\N}$ with color $i\in \{0,1\}$.  We show that $V(G)\in \nwd$.

	Let $F_n$ ($n\in\NN$) be pairwise-disjoint, finite ruling sets in $G$, and fix $s\in
    \N$.  Then there is some $n$ such that for all $t\in F_n$, $t>s$.  Let $u\in \N$ with $u >
    \max\set{t}{t\in F_n}$ such that $u$ extends $s$ and the $t$th bit of $u$ is $1-i$ for all $t\in F_n$.  This means that no vertex in $\lrang{u}$ is adjacent to any vertex in $F_n$ in color $i$.  Since $F_n$ is a ruling set, this implies that $V(G)\cap \lrang{u}$ is finite. Since $V(G)\cap \lrang{u}$ is finite, there exists $u'>u$ such that $u'$ extends $u$ and $V(G)\cap \lrang{u'}=\emptyset$ and thus $V(G)\in \nwd$.

    (\ref{nwd3}$\implies$\ref{nwd1}) Let $F$ be a finite subset of $V(G)$
    for which $G\sm F$ is $0$-ruled, and let $n = |F|$.  Fix a proper ideal
    $\mathcal{I}$ on $X$ containing the finite subsets of $X$, and let
    $\SU$ be an ultrafilter contained in $\mathcal{I}^+$ (which exists by \cref{fact:coideal-contains-ultrafilter}).
    Now we proceed exactly as in the proof of \cref{thm:ruling}.
%    Fix also an
%    $r$-coloring of $K_X$.  By the usual trick, we may find a set
%    $A\in\mathcal{U}$ and a color $i$ such that the $i$-colored neighborhood of
%    any vertex in $A$ is in $\mathcal{U}$.  Note that this implies that $A$,
%    with the $i$-colored edges, is a $0$-coruled graph.  Choose vertices
%    $v_1,\ldots,v_n\in A$ such that $\{v_p,v_q\}$ has color $i$ for all distinct
%    $p,q\in\{1,\ldots,n\}$.  Let $A'$ be the intersection of $A$ with the
%    $i$-colored neighborhoods of each $v_p$.  Then we may find an $i$-colored
%    monochromatic copy of $G\sm F$ with vertex set $A'$, and hence a
%    monochromatic copy of $G$ with vertex set $\{v_1,\ldots,v_n\}\cup A'$.
\end{proof}

We see that \cref{thm:ruling} is a special case of \cref{thm:nwd} since, in particular, \cref{thm:nwd} shows that every finitely-ruled graph $G$ is $\mathcal{Z}_0^+$-Ramsey.  Problem~\ref{prob:infruleddense} asks whether the converse is true; that is, if $G$ is $\mathcal{Z}_0^+$-Ramsey, is $G$ finitely-ruled?  Theorem~\ref{thm:nwd} might be viewed as evidence towards
this conclusion, since it shows that this is true at least for the coideal
$\nwd^+$ in place of $\mathcal{Z}_0^+$.  On the other hand, we might view
Theorem~\ref{thm:nwd} as evidence in the opposite direction, since one would
expect there to be some infinite graph $G$ which distinguishes the coideals
$\nwd^+$ and $\mathcal{Z}_0^+$ as $\nwd$ and $\mathcal{Z}_0$
have very different properties as ideals.  Of course, this is all just
speculation.

\subsection{Infinitely ruled graphs and relative density zero ideals}

Let $f : \NN\to\NN$ be a function.  The ideal $\mathcal{Z}_f$ is defined to be
the set of all $A\subseteq\NN$ such that $|A\cap \{1,\ldots,n\}| / f(n)\to 0$ as
$n\to \infty$.  The ideal $\mathcal{Z}_0$ is one example, where we take $f$ to
be the identity function.  The reader may check that $\mathcal{Z}_f \subseteq
\mathcal{Z}_g$ whenever $f \le g$, though of course the converse does not hold.

In Theorem \ref{rulinglog} we showed that if $G$ is infinitely ruled and the ruling sets grow slowly enough, then $G$ is not $\mathcal{Z}_0^+$-Ramsey (i.e.\ $G$ is not Ramsey dense).  In this section we will give an example of a family $\mathcal{G}$ of infinitely ruled graphs (where the size of the ruling sets may go to infinity at any prescribed rate, no matter how slowly) such that for all functions $f : \NN\to\NN$ satisfying $f(n) / n\to 0$ and for all $G\in \mathcal{G}$,  $G$ is $\mathcal{Z}_f^+$-Ramsey.

Let $T$ be a tree with a fixed root $r$.  Given vertices $u,v\in V(T)$, we say
that $v$ is an \emph{extension} of $u$ if $u$ lies on the unique path from $r$
to $v$.  We say that two vertices in $T$ are \emph{compatible} if one is an
extension of the other.  The \emph{compatibility graph of $(T,r)$}, $C_{T,r}$,
is the graph with vertex set $V(T)$ and edges $\{u,v\}$ for all compatible
vertices $u$ and $v$.  (This is sometimes called the \emph{downward closure} of $(T,r)$.)

An \emph{antichain} in $T$ is a set of pairwise-incompatible vertices, and we
call an antichain $A$ \emph{maximal} if there is no antichain $B$ such that
$A$ is a proper subset of $B$.  Note that every finite maximal antichain in $T$ is a ruling set in
$C_{T,r}$; in particular, if $T$ is locally finite, then the sets $R_n =
\set{v\in V(T)}{d_T(v,r) = n}$ form finite ruling sets in $C_{T,r}$.

We say that $T$ is \emph{perfect} if every vertex has at least two incompatible
extensions.  If $T$ is locally finite and perfect, then $|R_n|\to\infty$ as
$n\to\infty$, but the growth of this sequence may be arbitrarily slow.

\begin{theorem}\label{thm:ZfRamsey}
  Let $f : \NN\to\NN$ be any function satisfying $f(n) / n\to 0$, and let $T$ be
  any locally finite, perfect tree with fixed root $r$.  Then $C_{T,r}$
   is $\mathcal{Z}_f^+$-Ramsey.
\end{theorem}

\cref{thm:ZfRamsey} is immediately implied by the following two results.

\begin{proposition}\label{thm:Zf-multi}
  Suppose $f:\NN\to \NN$ satisfies $f(n) / n\to 0$.  Then for any finite coloring of $K_\NN$,
  there is a monochromatic complete multipartite subgraph, with finite parts,
  whose vertex set is in $\mathcal{Z}_f^+$.
\end{proposition}

\begin{proposition}\label{thm:multi}
  Let $T$ be a perfect tree with a fixed root $r$, and let
  $M$ be a complete, infinite multipartite graph, with finite parts.
  Then there is a copy of $C_{T,r}$ in $M$ which spans all but finitely many of the parts of $M$.
\end{proposition}

\begin{proof}[Proof of \cref{thm:Zf-multi}]
  Consider the ideal $\mathcal{I}_f$ consisting of all sets $A\subseteq\NN$ satisfying
  \[
    \sup_n |A\cap \{1,\ldots,n\}| / f(n) < \infty
  \]
  (Note that the assumption on $f$ implies that $\mathcal{I}_f$ is proper.)  Clearly
  $\mathcal{Z}_f \subseteq \mathcal{I}_f$.  We moreover note that if
  $A_1,A_2,\ldots$ is a sequence from $\mathcal{I}_f^+$ satisfying $A_1
  \supseteq A_2 \supseteq \cdots$, then there is a set $A\in\mathcal{I}_f^+$
  satisfying $|A\sm A_n| < \infty$ for all $n$. (For instance, $A$ may be
  constructed by letting $A = \bigcup_n A_n \cap [k_n, k_{n+1}]$, where $k_n$ is
  chosen recursively to satisfy $|A_n \cap [k_n, k_{n+1}]| / f(k_{n+1}) \ge n$.)

  Now fix an $r$-coloring $\chi$ of $K_\NN$.  We will recursively construct sets
  $A_1,A_2,\ldots \in\mathcal{I}_f^+$, and an $r$-coloring $\rho$ of $\NN$, such
  that $A_1\supseteq A_2 \supseteq \cdots$ and for all $n$ and $m\in A_n$,
  $\chi(\{n,m\}) = \rho(n)$.  This construction goes as follows.  First we set
  $A_1 = \NN$.  Now, given $A_n$, note that the sets $A_n \cap N_i(n)$, for $i =
  1,\ldots,r$, partition $A_n\sm\{n\}$, and hence at least one must be in
  $\mathcal{I}_f^+$ (by \cref{fact:part}).  We choose one to be $A_{n+1}$ and define $\rho(n)$ to be
  the associated color $i$.

  Now by the above-mentioned property of $\mathcal{I}_f^+$, we may find a single
  set $A'\in\mathcal{I}_f^+$ such that $|A'\sm A_n| < \infty$ for all $n$.  Now let $A=A'\cap \rho^{-1}(i)$ for some $i\in [r]$ such that $A\in \mathcal{I}_f^+$.  Hence, for all $n\in
  A$ there are only finitely-many $m\in A$ for which $\chi(\{n,m\}) \neq i$.  In other words, the graph $K_i[A]$ consisting of edges of color $i$ induced on vertex set $A$ has the property that every vertex has cofinite degree.
  This allows us to construct, recursively, a sequence $b_0 < b_1 < \cdots$ such that for all $n, m\in A$ with $n \le b_k$ and $m \ge b_{k+1}$, $\chi(\{n,m\}) = i$.
  Let
  \[
    A_0^* = A\cap \bigcup_{k=0}^\infty [b_{2k},b_{2k+1})
  \]
  and $A_1^* = A\sm A_0$.  Then each of $A_0^*$ and $A_1^*$ is the vertex set of a monochromatic multipartite graph with finite parts, and moreover at least one is in
  $\mathcal{I}_f^+$, and hence $\mathcal{Z}_f^+$ (by Fact \ref{fact:subset}, since $\mathcal{Z}_f\subseteq \mathcal{I}_f$).
\end{proof}

\begin{proof}[Proof of \cref{thm:multi}]
Let $r_1:=r$ and let $P=r_1r_2\dots r_k$ be the shortest path from $r_1$ to a vertex $r_k$ such that $r_k$ has at least two successors (if $r_1$ itself has two successors, then $r_k=r_1$ and $P=r_1$ is just a trivial path).  Note that every vertex $r_i$ on the path $P$ is a maximal antichain (and thus a ruling set in $C_{T,r}$).  Also every vertex $v\in V(C_{T,r})\setminus V(P)$ is part of a maximal infinite independent set which we denote $I(v)$.

Let $K:=K_{1,\dots,1, \N,\N, \dots}$ be the complete multipartite graph with $k$ parts of order 1 and infinitely many infinite parts.  Clearly $K$ can be embedded into $M$ in such that way that $K$ spans all but finitely many parts of $M$.  We will show that $C_{T,r}$ can be surjectively embedded into $K$ which will then complete the proof.

First we embed the path $P=r_1\dots r_k$ into the parts of order 1.  Let $v_1, v_2, \dots$ be an enumeration of the remaining vertices of $K$.  Note that this ordering induces an ordering $V^1, V^2, \dots$ of the infinite parts themselves (meaning that if $v_n\in V^j$, then $\{v_1, v_2,  \dots, v_n\}\subseteq \bigcup_{i=1}^jV_i$) and an ordering $v_1^i, v_2^i$ of each $V^i$.  Finally, let $u_1, u_2, \dots$ be an enumeration of $V(T)\setminus V(P)$ such that $T[\{r_1, \dots, r_k, u_1, \dots, u_i\}]$ is connected for all $i\geq 1$.

Initially we set $f(u_1)=v_1$ (where $v_1\in V^1$) and then we repeat the following two steps.

		\par\smallskip
		\noindent\underline{Step 1.} Let $v^i_j\in V^i$ be the smallest vertex in $V(K)\setminus \ran f$.  If $j=1$ (i.e.\ $V^i\cap \ran f=\emptyset$, move to Step 2.  Otherwise, let $u\in \dom f$ such that $f(u)\in V^i$.  Now let $u'\in I(u)\setminus \dom f$ and set $f(u')=v^i_j$.

		\par\smallskip
		\noindent\underline{Step 2.}
		Let $m$ be the largest index such that $u_{m-1}\in \dom f$ and let $U'=\{u_1', u_2', \dots, u_\ell'\}:=\{u_1, \dots, u_m\}\setminus \dom f$.  Let $n$ be the largest index such that $V^n\cap \ran f\neq \emptyset$.  Now for all $i\in [\ell]$, set $f(u_i')=v_1^{n+i}$ (where $v_1^{n+i}$ is the first vertex in $V^{n+i}$).  Note that for all $i\geq 1$, $v_1^{n+i}$ is adjacent to every vertex in $\ran f$.

		\par\smallskip

At the end of each instance of Step 1 and Step 2, we have covered the first available vertex in $V(K)\setminus \ran f$ and we have embedded an entire interval $\{u_1, u_2, \dots, u_m\}$ in the ordering of $V(T)$, including the first available vertex in $V(T)\setminus \dom f$.  Thus we have defined a surjective embedding of $C_{T,r}$ into $K$, which completes the proof.
\end{proof}

In the proof of \cref{thm:Zf-multi}, we have a graph with vertex set $A\in \CI_f^+$ in which every vertex has cofinite degree.  We use this to show that $A$ can be partitioned into two infinite complete multipartite graphs with all parts finite, one of which, call it $M$, must have vertex set in $\CI_f^+$.  We then show that if $T$ is perfect tree with fixed root $r$, then $C_{T,r}$ can be embedded into $M$.  This raises the following two questions.

\begin{problem}\label{q:cofin}~
\begin{enumerate}
\item Characterize all graphs which can be cofinitely embedded into every graph in which every vertex has cofinite degree.

\item Characterize all graphs which can be cofinitely embedded into every infinite complete multipartite graph with finite parts.
\end{enumerate}
\end{problem}

\section{Conclusion and open problems}\label{sec:end}

\subsection{Graphs of bounded chromatic number/maximum degree/degeneracy}

Let $G$ be a graph with $\Delta:=\Delta(G)\geq 2$.  We know that $2\leq \chi(G)\leq \Delta(G)+1$ and we proved that $\frac{1}{2(\Delta-1)}\leq \Rd(G)\leq \frac{1}{\chi(G)-1}$.  It would be interesting to know whether these bounds can be improved in general.  From \cref{ex:trees}.(ii), we know that the bound in \cref{thm:locally-finite}.(i) cannot be improved without further restrictions.

\begin{problem}
If possible, improve the bounds in \cref{thm:locally-finite}.(ii),(iii).
\end{problem}

We make the following conjecture which would imply \cref{qu:degenerate} in the case where $G$ has no finite dominating set (see \cref{q:ruldeg} and \cref{prob:2dir}).

\begin{conjecture}\label{con:Hd}
Let $d\in \N$ and let $H_d$ be the graph defined in \cref{prop:constr}.  There exists $c=c(d)>0$ such that $\Rd(H_d)\geq c$.  More weakly, $\Rd(H_d)>0$.
\end{conjecture}

We know that \cref{con:Hd} is true when $d=1$ because $H_1=T_\infty$ and $\Rd(T_\infty)=1/2$.  Even solving the conjecture when $d=2$ would be a big step forward.

Note that for all $r\geq 3$, \cref{ex:3ormore} shows that $\Rd_r(H_d)=0$ for all $d\in \N$.  This stands in contrast to the finite case where Lee's proof \cite{L} shows that \cref{conj:linRam2} holds for any number of colors.  On the other hand, if $G$ is $d$-degenerate and locally finite, then \cref{thm:locally-finite}.(iii) implies that $\Rd_r(G)\geq 1/r^{dr}$.  So for more than 2 colors, the only interesting case is $d$-degenerate graphs with infinitely many vertices of finite degree and some vertices of infinite degree.

\subsection{Ramsey-dense graphs and graphs with positive upper Ramsey density}

We know that every 0-ruled graph is Ramsey-dense and we know that there exist 0-ruled graphs $G$ with $\Rd(G)=0$, but all such graphs $G$ that we know of have $\chi(G)=\infty$ (see \cref{cor:Rado-ud-0}).  So we ask the following question.

\begin{problem}
Does there exist a graph $G$ which is 0-ruled and $\chi(G)<\infty$, but $\Rd(G)=0$?  For instance, is $\Rd(\CR_2)=0$? (where $\CR_2$ is the bipartite Rado graph)
\end{problem}

Note that for $r\geq 3$, \cref{ex:3ormore} provides examples of 0-ruled graphs $G$ with $\chi(G)<\infty$, but $\Rd_r(G)=0$.

In \cref{prob:infruleddense} we ask if there are Ramsey-dense graphs with $\rul(G)=\infty$.  A more expansive series of questions is the following.

\begin{problem}~
\begin{enumerate}
\item Characterize all Ramsey-dense graphs.

\item Characterize graphs $G$ with $\Rd(G)>0$.

\item Characterize graphs $G$ which are Ramsey-dense, but $\Rd(G)=0$.
\end{enumerate}
\end{problem}

In \cref{thm:ZfRamsey} we proved that every graph in a certain class of graphs with infinite ruling number is $\mathcal{Z}_f^+$-Ramsey. So we raise the following problem.  Also see \cref{q:cofin}.

\begin{problem}
Characterize all graphs $G$ having the property that for all functions $f: \NN\to\NN$ satisfying $f(n) / n\to 0$, $G$ is $\mathcal{Z}_f^+$-Ramsey.
\end{problem}

%%%%%%%%%%%%%%%%%%%%%%%%%%%%%%%%%%%%%%%%%%%%%%%%%%%%%%%%%%%%%%%%%%%%%%%%%%%%%%%%%
\subsection{Ramsey-lower-dense and Ramsey-cofinite graphs}

One of the main results in the paper is a characterization of all Ramsey-cofinite forests.  The most interesting open problem here is the following (c.f.\ \cref{q:specific}).

\begin{problem}
Characterize all Ramsey-cofinite graphs.
\end{problem}

In \cref{thm:cofinitetree}, we proved that every forest is either Ramsey-cofinite or is not Ramsey-lower-dense.  We suspect that this is true of every graph (in \cite[Problem 8.10]{DM} we asked the weaker question of whether $\Rd(G)<1$ implies that $G$ has Ramsey-lower-density 0).

\begin{conjecture}
For every graph $G$, if $G$ is not Ramsey-cofinite, then $G$ is not Ramsey-lower-dense.
\end{conjecture}

%
%Another case of interest is the class of graphs $G$ with finitely many isolated vertices such that for each $\Delta\geq 1$ there are infinitely-many vertices with degree $\Delta$.
%%%%%%%%%%%%%%%%%%%%%%%%%%%%%%%%%%%%%%%%%%%%%%%%%%%%%%%%%%%%%%%%%%%%%%%%%%%%%%%%
\subsection{Ramsey-upper-density of trees}
We showed that $\Rd(T)\geq 1/2$ for every infinite tree $T$ and we showed that this result is tight for some trees such as $T_{\infty}$.
Lamaison \cite{Lamaison2020} obtained sharp results on Ramsey upper densities of locally finite trees.

It would be interesting to extend some of these results to more colors.  By \cref{ex:3ormore}, we know that for all $r\geq 3$, if $T$ is a tree with finitely many vertices of finite degree, then $\Rd_r(T)= 0$.  On the other hand, we know from \cref{thm:locally-finite}, that if $T$ is an infinite, locally finite tree, then $\Rd_r(T)\geq 1/r$ (which is best possible for some trees); or more generally, if $T$ is an infinite, one-way $k$-locally finite tree for some $k\geq 2$, then $\Rd_r(T)\geq 1/r^{(k-2)r+1}$.  So for more than 2 colors one should focus on trees with infinitely many vertices of finite degree which are not one-way $k$-locally finite for any $k\geq 2$.

\subsection{Bipartite graphs}
We conjectured (\cref{conj:bip-paths}) that the vertices of every $r$-colored $K_{\NN, \NN}$ can be partitioned into a finite set and at most $r$ monochromatic paths. We proved this for $r=2$ and mentioned how a result of Day and Lo \cite{DL} implies that if \cref{conj:bip-paths} is true, then for all $r\geq 3$, $\Rd_r(P_\infty)\geq \frac{1}{r-1}$ (which is now an open problem for $r\geq 5$).

\medskip

\noindent
\tbf{Acknowledgments:} We thank Andr\'as Gy\'arf\'as for discussions which led to \cref{ex:3ormore}.  We are grateful to the referees for their careful reading of the paper and many helpful comments.  

%%%%%%%%%%%%%%%%%%%%%%%%%%%%%%%%%%%%%%%%%%%%%%%%%%%%%%%%%%%%%%%%%%%%%%%%%%%%%%%%%
%%%%%%%%%%%%%%%%%%%%%%%%%%%%%%%%%%%%%%%%%%%%%%%%%%%%%%%%%%%%%%%%%%%%%%%%%%%%%%%%%
\bibliographystyle{amsplain-abbr}
\bibliography{bib}
\end{document}